\documentclass{amsart}

\usepackage{xurl}
\usepackage{url}
\usepackage{mathtools}
\usepackage[square,numbers]{natbib}
\usepackage{float}
\restylefloat{table}
\usepackage[utf8]{inputenc}
\usepackage[toc,page]{appendix}
\usepackage{amsmath,amsthm,amssymb} 
\usepackage{graphicx}
\usepackage[colorinlistoftodos,prependcaption]{todonotes}
\usepackage{xargs}
\usepackage{multicol}
\newcommandx{\unsure}[2][1=]{\todo[linecolor=blue,backgroundcolor=blue!25,bordercolor=blue,#1]{#2}}
\usepackage[section]{placeins}
\usepackage{float}
\usepackage{comment}
\usepackage{xcolor}
\usepackage{graphicx}
\usepackage{dcolumn}
\usepackage{bm}
\usepackage{amsaddr}

\usepackage{tabularray,graphicx,xcolor}
\UseTblrLibrary{diagbox}

\newmuskip\pFqmuskip
\newcommand*\pFq[6][8]{
	\begingrou
	\pFqmuskip=#1mu\relax
	\mathcode`\.=\string"8000
	\begingroup\lccode`\~=`\,
	\lowercase{\endgroup\let~}\pFqcomma
	{}_{\,#2}F_{\,#3}{\left[\genfrac..{0pt}{}{\,#4}{\,#5};#6\right]}
	\endgroup
}
\newcommand{\pFqcomma}{\mskip\pFqmuskip}

\makeatletter
\renewenvironment{proof}[1][\proofname]{%
  \par\pushQED{\qed}\normalfont%
  \topsep6\p@\@plus6\p@\relax
  \trivlist
  \item[\hskip\labelsep\bfseries #1\@addpunct{.}]%
}{%
  \popQED\endtrivlist\@endpefalse
}
\makeatother

\newtheorem{theorem}{Theorem}[section]
\newtheorem{corollary}{Corollary}[section]

\newtheorem{lemma}{Lemma}[section]

\newcommand{\bbint}[2]{\ensuremath{\;\backslash\!\!\!\!\backslash\!\!\!\!\!\int_{#1}^{#2}}}

\numberwithin{equation}{section}

\usepackage{hyperref} 
\hypersetup{colorlinks=true, citecolor=blue, linkcolor=blue, urlcolor=blue}
\DeclareUnicodeCharacter{2212}{-}

\begin{document}
\title[Finite-part integrals with logarithmic singularities]{Contour Integral Representations of Finite-part Integrals with Logarithmic Singularities}

\author{Reynaldo P. Ylanan \& Eric A. Galapon}
\address{Theoretical Physics Group, National Institute of Physics, University of the Philippines, Diliman, Quezon City, 1101 Philippines}
\email{eagalapon@up.edu.ph}
\date{\today}


\begin{abstract}
The integral $\int_0^a f(t) t^{-s} \mathrm{d}t$ diverges for $\text{Re}(s) \geq \lambda + 1$, where $\lambda$ is the order of the first non-vanishing derivative of $f(t)$ at the origin. With the assumption that $f(t)$ is analytic at the origin, the finite-part of the divergent integral assumes the contour integral representation of the form $\bbint{0}{a} f(t) t^{-s} \mathrm{d}t =   \int_C f(z) z^{-s} G(z)  \mathrm{d}z$ where $G(z)$ depends on whether $z=0$ constitutes a pole or a branch point singularity of $z^{-s}$ [E. A. Galapon, \textit{Proc. R. Soc.}, \textbf{A 473} (2017), no. 2197, 20160567.]. In this paper, we extend these representations to accommodate logarithmic singularities of arbitrary order $n \in \mathbb{N}$, specifically for $\bbint{0}{a} f(t) t^{-s} \ln^n t \, \mathrm{d}t$. We then demonstrate the utility of the representations in the numerical evaluation of finite-part integrals and their use in determining the finite parts of non-Mellin-type divergent integrals---those which exhibit singular behavior at the origin but lack a well-defined Mellin transform. Finally, these representations provide a closed-form evaluation of the Stieltjes transform $\int_0^a k(t) \ln^n t \left( t^\nu (\omega^2 + t^2) \right)^{-1} \mathrm{d}t$ in terms of finite-part integrals, from which the dominant asymptotic behavior is readily extracted for vanishingly small values of the parameter $\omega$.
\end{abstract}

\maketitle


\section{Introduction}

Finite-part integration is a method of solving convergent integrals by evaluating the finite-part of the divergent integral induced from the convergent integral itself \cite{galaponMissing, galapon2023, villanueva2021finite}. The technique is applied for the exact evaluation of the  Stieltjes transform \cite{galapon2023, tica, tica2, TicaGalapon2023},
    \begin{align}
        \label{stiel-trans}
        \mathcal{S}_f (\omega) = \int_0^a \frac{f(t)}{(\omega + t)} \, \mathrm{d}t.
    \end{align}
Divergent integrals are introduced by expressing the kernel function as an infinite series
    \begin{align}
        \label{kern-exp}
        \frac{1}{\omega + t} = \sum^\infty_{s=0} \frac{(-1)^s}{t^{s+1}}\omega^s, \qquad \left| \frac{\omega}{t}\right| < 1 ,
    \end{align}
where the expansion is about \(\omega =0 \). The divergence arises from the formal interchange in the order of integration and summation without conforming to the uniformity conditions of the expansion in the given limit of integration. The right-hand side of equation \eqref{stiel-trans} becomes
    \begin{align}
        \label{stiel-trans-1}
        \sum^\infty_{s=0} (-\omega)^s \int_0^a \frac{f(t)}{t^{s+1}} \, \mathrm{d}t.
    \end{align}
The integral inside the summation is expected to diverge if the function \(f(t)\), or any of its derivatives, does not vanish at \(t=0\). In simple terms, equation \eqref{stiel-trans-1} becomes a series in which each term is divergent. To use the equality in equation \eqref{stiel-trans}, the divergent integrals can be replaced by their finite-parts plus the missing terms that can be recovered by lifting the integration to the complex plane. It turns out that the missing terms are encoded in the singularity of the kernel function \cite{galaponMissing}. The finite-part of the divergent integral needed in this context can be defined by replacing the lower limit of integration with some positive number \(\epsilon\) that is \(0<\epsilon<a\) and expressing the solution of the integral as
    \begin{align}
        \label{stiel-trans-1-1}
        \int_\epsilon^a \frac{f(t)}{t^{s+1}} \, \mathrm{d}t = C_{\epsilon} + D_{\epsilon}.
    \end{align}
where \(C_\epsilon\) converges and \(D_\epsilon\) diverges as \(\epsilon \to 0\). The finite-part of the divergent integral in equation \eqref{stiel-trans-1} can now be denoted and defined as \cite{galapon2023}
    \begin{align}
        \label{FPI}
        \bbint{0}{a} \frac{f(t)}{t^\lambda} \mathrm{d}t = \lim_{\epsilon \to 0} C_\epsilon = \lim_{\epsilon \to 0} \left( \int_\epsilon^a\frac{f(t)}{t^\lambda} \mathrm{d}t - D_\epsilon\right).
    \end{align}

Under certain conditions, the finite-part integral can be represented by the analytic continuation of the Mellin transform beyond its strip of analyticity. The values of the parameter \(\lambda\) in equation~\eqref{FPI} determine the relationship between these two operations. When \(\lambda\) is a natural number, the evaluation of the finite-part integral requires the concept of the regularized limit. In the Laurent series expansion of a function about a singular point, the regularized limit corresponds to the coefficient of the zero-order term, i.e., the constant term.  

To compute the regularized limit for functions with poles of order \(n\), a matrix formulation of size \((n+1) \times (n+1)\) is employed. This matrix arises from the algebraic division of power series and is discussed in detail in~\cite{galapon2023}, \cite[p. 18, 0.313]{Gradshteyn-Ryzhik}. While this matrix-based approach is rigorous, it becomes increasingly tedious for functions with higher-order poles. One of the aims of this paper is to derive a closed-form expression for the regularized limit. Such a representation is expected to provide a more practical and efficient method for evaluating regularized limits involving poles of arbitrary order.

Contour integration has been shown to be an effective method complementary to finite-part integration in the evaluation of convergent integrals. Using the simple contour of integration illustrated in Figure~\ref{Intro-fig-01}, the finite-part integral appearing in equation~\eqref{stiel-trans-1} can be expressed as a contour integral~\cite{galaponMissing}:
\begin{figure}[t]
    \centering
    \includegraphics[width=6cm]{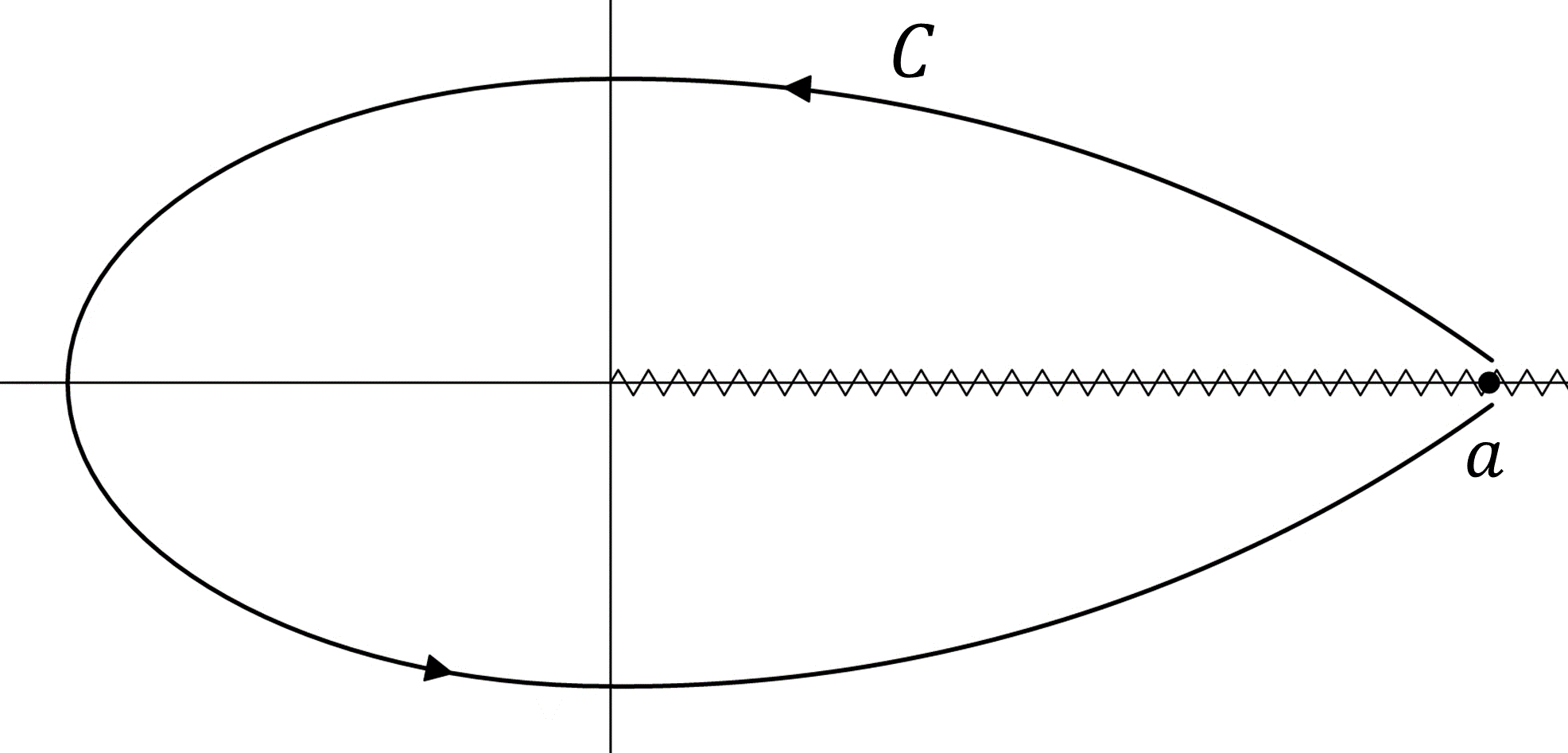}
    \caption{Contour of integration \(C\). The branch cut for \(z^\lambda\) lies along the positive real axis.}
    \label{Intro-fig-01}
\end{figure}
\begin{equation}
    \label{cont-int-rep-01}
    \bbint{0}{a} \frac{k(t)}{t^\lambda} \mathrm{d}t = \int_C \frac{k(z)}{z^\lambda} G(z) \mathrm{d}z.
\end{equation}
where $G(z)$ is dertermined by whether $z=0$ is a pole or a branch point of $z^{-\lambda}$. For a pole singularity, i.e., when $\lambda \in \mathbb{Z}_0^+$, 
\begin{equation}
    \label{cont-int-rep-001}
    G(z) = \frac{1}{2\pi i} (\log z - \pi i).
\end{equation}
On the other hand, for a branch point, i.e. $\lambda \notin \mathbb{Z}$, it is given by
\begin{equation}
    \label{cont-int-rep-002}
    G(z) = \frac{1}{(e^{-2\pi i \lambda} - 1)} .
\end{equation}
This expression serves as the starting point for the development of contour integral representations analyzed in this paper. A notable limitation of equation~\eqref{cont-int-rep-01} is that it does not account for the presence of a logarithmic singularity. Another objective of this work is to extend the contour integral representation of the finite-part integral to accommodate integrands that exhibit logarithmic singularities of arbitrary order.

The contour integral representation extends the domain of the Mellin transform beyond its strip of analyticity by leveraging its connection with the finite-part integral. While the Mellin transform may not always yield a well-defined result, even under analytic continuation, the finite-part integral is expected to exist or, at the very least, to evaluate to zero. As demonstrated in~\cite{galapon2023}, the finite-part integral can exist independently of the Mellin transform. The crucial insight lies in the use of the contour integral representation of the finite-part integral, which enables the evaluation of otherwise ill-defined expressions.

An additional advantage of the contour integral approach is its utility in the numerical computation of finite-part integrals. By employing contour deformation techniques, one can determine the desired solution to the given integral problem with greater flexibility and precision.

Finally, the last objective of this paper is to demonstrate the application of finite-part integration in evaluating the Stieltjes transform of the form
\begin{equation}
    \int_0^\infty \frac{k(t) \ln^n{t}}{t^\nu (\omega^2 + t^2)} \; \mathrm{d}t,
\end{equation}
This type of integral has a potential connection to the Euler–Heisenberg effective Lagrangian~\cite{Tica-Blancas-Galapon-2025-1}. The Euler–Heisenberg Lagrangian represents a non-linear correction to the classical Maxwell Lagrangian, and is used in quantum electrodynamics to describe the phenomenon of vacuum polarization in the presence of an external electromagnetic field~\cite{dune-01, dunne2012heisenberg}.

The rest of this paper is organized as follows. In Section~\ref{regularized_lim}, we introduce the concept of the regularized limit and present a new formulation that incorporates the ideas of integer partitions and the reciprocal of an infinite series. Section~\ref{contour-integral-fpi} develops the contour integral representation of the finite-part integral in the presence of logarithmic singularities of arbitrary order. This section also demonstrates how the representation can be used to compute the numerical values of finite-part integrals and to establish their existence beyond the framework of the Mellin transform. In Section~\ref{Stieljes-Transform-Section}, we apply the developed techniques to evaluate the generalized Stieltjes transform in cases involving logarithmic singularities.


\section{The Regularized Limit}
\label{regularized_lim}

Let us consider a function \( w(\lambda) \) that is meromorphic on a domain \( D \subseteq \mathbb{C} \), meaning it remains analytic on \( D \) except at a finite number of isolated singularities. These singularities may be removable, essential, or poles. Suppose \( \lambda_0 \) is an interior point of \( D \) and serves as an isolated singularity of \( w(\lambda) \). Then \( w(\lambda) \) admits the following Laurent series expansion around \( \lambda_0 \):
\begin{align}
    \label{eq2.1}
    \begin{split}
        w(\lambda) & = \sum_{k=-\infty}^\infty a_{k}(\lambda - \lambda_0)^k = \sum_{k=1}^\infty \frac{a_{-k}}{(\lambda - \lambda_0)^k} + \sum_{k=0}^\infty a_{k}(\lambda - \lambda_0)^k,
    \end{split}
\end{align}
which holds in the deleted neighborhood \( D' = D(r, \lambda_0) \setminus \{ \lambda_0 \} \), where \( D(r, \lambda_0) \) denotes an open disk of radius \( r \) centered at \( \lambda_0 \). The first sum in equation \eqref{eq2.1} is called the principal part, denoted \( w_P(\lambda \mid \lambda_0) \), and the second sum is called the regular part, denoted \( w_R(\lambda \mid \lambda_0) \). The principal part diverges as \( \lambda \to \lambda_0 \).

The regularized limit of \( w(\lambda) \) at the singular point \( \lambda_0 \) is defined as
\begin{align}
    \label{eq2.2}
    \lim^\times_{\lambda \to \lambda_0} w(\lambda) = \lim_{\lambda \to \lambda_0} \left[ w(\lambda) - w_P(\lambda \mid \lambda_0) \right].
\end{align}
This regularized limit extracts the finite and constant part of the function as \( \lambda \to \lambda_0 \), yielding
\begin{equation}
    \lim^\times_{\lambda \to \lambda_0} w(\lambda) = a_0,
\end{equation}
where \( a_0 \) is the constant term in the expansion of \( w(\lambda) \), as given in equation~\eqref{eq2.1}. We can represent it using a contour integral:
\begin{align}
    \label{reg-lim-cont-int}
    \lim^\times_{\lambda \to \lambda_0} w(\lambda) = \frac{1}{2 \pi i} \oint_{|\lambda - \lambda_0| = \rho} \frac{w(\lambda)}{(\lambda - \lambda_0)} \, \mathrm{d}\lambda,
\end{align}
where \( \rho \) is sufficiently small. This result follows directly from the integral representation of the constant coefficient \( a_0 \) in the Laurent expansion \cite{galapon2023}.

One important property of the regularized limit used in this paper is its linearity. Given an infinite sequence of functions \(\{w_k (\lambda)\}^\infty_{k=0}\) of the complex variable \(\lambda\), each analytic in some region \(D\), except possibly at isolated singularities. Suppose \( \lambda_0 \) is an interior point of \( D \), and that the infinite series \(\sum_{k=0}^\infty w_k(\lambda)\) converges uniformly in some deleted neighborhood, \( D' = D(r, \lambda_0) \setminus \{ \lambda_0 \} \), of \(\lambda_0\). Then
\begin{align}
    \label{lin-reg-lim}
    \lim^\times_{\lambda \to \lambda_0} \sum_{k=0}^\infty w_k(\lambda) = \sum_{k=0}^\infty \lim^\times_{\lambda \to \lambda_0} w_k(\lambda).
\end{align}
This means that the regularized limit can be distributed over the summation given the right conditions. This relation is presented in \cite{galapon2023} as one of the theorems under the discussion of the regularized limit. The proof is straightforward: since all \(w_k (\lambda)\) are analytic in \(D'\), the infinite series on the left-hand side of equation \eqref{lin-reg-lim} may replace \(w(\lambda)\) in the right-hand side of the contour integral representation \eqref{reg-lim-cont-int}, and term-by-term integration is allowed. This leads directly to the right-hand side of equation \eqref{lin-reg-lim}.

Consider a point \(\lambda_0\) that is a pole of \(w(\lambda)\) of order \(n\). Suppose \( w(\lambda) \) can be written as a rational function \( f(\lambda)/g(\lambda) \), where \( f(\lambda_0) \neq 0 \) and \( g(\lambda_0) = 0 \), so that the singularity arises from the vanishing of the denominator. In this case, the regularized limit of \( w(\lambda) \) at \(\lambda_0\) is denoted by
\begin{equation}
    \label{eq2.3}
    \lim^\times_{\lambda \to \lambda_0} \frac{f(\lambda)}{g(\lambda)}.
\end{equation}
The explicit formula for computing this limit is given by \cite{galapon2023}
\begin{align}
    \label{reg-lim-old}
    \lim^\times_{\lambda \to \lambda_0} \frac{f(\lambda)}{g(\lambda)} = \left( \frac{n!}{g^{(n)}(\lambda_0)} \right)^{n+1} \det \Delta^{(n)}(\lambda_0),
\end{align}
where \( \Delta^{(n)}(\lambda_0) \) is an \( (n+1) \times (n+1) \) matrix of the form
\begin{align}
    \Delta^{(n)}(\lambda_0) = 
    \begin{bmatrix}
        \frac{g^{(n)}(\lambda_0)}{n!} & 0 & 0 & \cdots & \frac{f(\lambda_0)}{0!} \\
        \frac{g^{(n+1)}(\lambda_0)}{(n+1)!} & \frac{g^{(n)}(\lambda_0)}{n!} & 0 & \cdots & \frac{f^{(1)}(\lambda_0)}{1!} \\
        \frac{g^{(n+2)}(\lambda_0)}{(n+2)!} & \frac{g^{(n+1)}(\lambda_0)}{(n+1)!} & \frac{g^{(n)}(\lambda_0)}{n!} & \cdots & \frac{f^{(2)}(\lambda_0)}{2!} \\
        \vdots & \vdots & \vdots & \ddots & \vdots \\
        \frac{g^{(2n)}(\lambda_0)}{(2n)!} & \frac{g^{(2n-1)}(\lambda_0)}{(2n-1)!} & \frac{g^{(2n-2)}(\lambda_0)}{(2n-2)!} & \cdots & \frac{f^{(n)}(\lambda_0)}{n!}
    \end{bmatrix}.
\end{align}

We find that equation~\eqref{reg-lim-old} is reliable for generating explicit formulas; however, it remains difficult to manipulate and incorporate into broader mathematical frameworks, particularly when attempting to generalize the contour integral representation of the finite-part integral, the central focus of this paper. To address this limitation, we present the regularized limit as an explicit summation over indices in Theorem~\ref{theorem2.1}.

\begin{lemma}[Galapon, \cite{galapon2023}]
    \label{lemma2.1}
    Consider a function \(v(\lambda)\) that is analytic at \(\lambda = \lambda_o\). Then,
        \begin{align}
            \label{lemma2.1-001}
            \lim^\times_{\lambda \to \lambda_o} \frac{v(\lambda)}{(\lambda -\lambda_0)^n} = \frac{v^{(n)}(\lambda_0)}{n!}, 
        \end{align}
    for any positive arbitrary n.
\end{lemma}

\begin{lemma}[Salem, \cite{salem}]
\label{lemmaSk}
    Given a function \( y(\lambda) \) defined as the reciprocal of an infinite series,
    \begin{align}
    \label{y_rec-inf-ser}
        y(\lambda) = \frac{1}{\sum_{\ell=0}^\infty a_\ell (\lambda - \lambda_0)^\ell},
    \end{align}
    where \( a_\ell \) are the coefficients of a power series centered at \( \lambda_0 \) and \(a_0 \neq 0\), the \(k^{\text{th}}\) derivative of \( y(\lambda) \) evaluated at \( \lambda_0 \), denoted by \( S_k \), is given by
    \begin{align}
    \label{Sk}
        S_k = \frac{k!}{a_0} \sum_{i=1}^{p(k)}  J^{(k)}_i ! \, \prod_{r=1}^{k} \frac{1}{m_{ri} !} \left( - \frac{a_r}{a_0} \right)^{m_{ri}},
    \end{align}
    where \( S_0 = 1 / a_0 \), \( p(k) \) is the number of partitions of the positive integer \( k \), and \( J_i^{(k)} \) is given by
    \begin{equation}
    \label{Sk1}
        J^{(k)}_i = \sum_{\ell=1}^k m_{\ell i},
    \end{equation}
    where the \( m_{\ell i} \) are integers satisfying the identity
    \begin{align}
        k = m_{1i} + 2m_{2i} + \cdots + k m_{ki}.
    \end{align}
\end{lemma}

\begin{lemma}[Salem, \cite{salem}]
\label{rep-of-an-inf-ser}
Given a function \( y(\lambda) \), defined as the reciprocal of the infinite series in equation~\eqref{y_rec-inf-ser}, it can be expressed as
\begin{align}
    y(\lambda) = \sum_{n=0}^\infty A_n (\lambda - \lambda_0)^n,
\end{align}
where
\begin{align}
    A_n = \frac{1}{a_0} \sum_{i=1}^{p(n)} J^{(n)}_i ! \, \prod_{r=1}^{n} \frac{1}{m_{ri}!} \left( -\frac{a_r}{a_0} \right)^{m_{ri}},
\end{align}
and \( a_0 \neq 0 \).
\end{lemma}

\begin{theorem}
    \label{theorem2.1}
    Let \(f(\lambda)\) be anaytic at \(\lambda_0\) with \(f(\lambda_0) \neq 0\) and \(\lambda_0\) be a zero of \(g(\lambda)\) of order n, then
    \begin{align}
    \begin{split}
        \label{theorem2.1-001}
        \lim^\times_{\lambda \to \lambda_o} \frac{f(\lambda)}{g(\lambda)} = & \frac{f^{(n)} (\lambda_o) }{g^{(n)} (\lambda_o) } + \frac{1}{g^{(n)} (\lambda_o)} \sum^{n}_{k=1} k! \binom{n}{k} \,  f^{(n-k)}(\lambda_o) \\
        & \hspace{1.5cm} \times \sum^{p(k)}_{i=1} J_i^{(k)} ! \, \prod^{k}_{r=1} \frac{1}{m_{ri} !} \Biggr[ - \frac{n!}{(r+n)!} \frac{g^{(r+n)}( \lambda_o)}{g^{(n)}( \lambda_o)} \Biggl]^{m_{ri}},
    \end{split}
    \end{align}
    where \( p(k) \) denotes the number of partitions of \( k \in \mathbb{N} \), and \( J^{(k)}_i \) and \( m_{\ell i} \) are defined in equations~\eqref{Sk} and~\eqref{Sk1}, respectively.
\end{theorem}

\begin{proof}
Since \(g(\lambda)\) has a zero of order \(n\) at \(\lambda_0\), we express it as
\begin{align}
    \label{g(lambda)-0}
    g(\lambda) = (\lambda-\lambda_0)^n \, h(\lambda),
\end{align}
where \(h(\lambda_0) \neq 0\). Then, rewrite the left-hand side of equation \eqref{theorem2.1-001} as
\begin{equation}
    \label{theorem2.1-003}
    \lim^\times_{\lambda \to \lambda_o} \frac{f(\lambda)}{g(\lambda)} = \lim^\times_{\lambda \to \lambda_o} \frac{f(\lambda)}{(\lambda - \lambda_0)^n h(\lambda)}.
\end{equation}
Apply Lemma~\ref{lemma2.1} to evaluate equation \eqref{theorem2.1-003} and obtain
\begin{equation}
    \label{theorem2.1-004}
    \lim^\times_{\lambda \to \lambda_o} \frac{f(\lambda)}{g(\lambda)} = \frac{1}{n!} \frac{d^n}{d\lambda^n} \left[ \frac{f(\lambda)}{h(\lambda)} \right]_{\lambda = \lambda_0}.
\end{equation}
We then use the Leibniz rule for differentiation and simplify to get
\begin{equation}
    \label{theorem2.1-005}
    \lim^\times_{\lambda \to \lambda_o} \frac{f(\lambda)}{g(\lambda)} = \frac{1}{n!}  \left[ \, \frac{f^{(n)}(\lambda_0)}{h(\lambda_0)} + \sum_{k=1}^n \binom{n}{k} f^{(n-k)}(\lambda_0) \, \frac{d^k}{d\lambda^k} \left[ \frac{1}{h(\lambda)}  \right]_{\lambda = \lambda_0} \, \right] .
\end{equation}

From equation \eqref{theorem2.1-005}, It is apparent that evaluating the regularized limit of \(w(\lambda)\) reduces to solving the \(n^{\text{th}}\) derivative of \(1/h(\lambda)\). Let us represent \(h(\lambda)\) as a power series, so that its reciprocal becomes
\begin{equation}
    \label{1/h(lambda)}
    \frac{1}{h(\lambda)} = \frac{1}{\sum_{\ell=0}^\infty a_\ell (\lambda - \lambda_0)^\ell},
\end{equation}
with \(a_0 \neq 0\). By applying Lemma~\ref{rep-of-an-inf-ser}, we represent the reciprocal of the infinite series in equation \eqref{1/h(lambda)} as
\begin{align}
    \frac{1}{\sum_{\ell=0}^\infty a_\ell (\lambda - \lambda_0)^\ell} =  \sum_{\ell=0}^\infty \;A_\ell (\lambda - \lambda_0)^\ell,
\end{align}
where
\begin{align}
    A_\ell = \frac{1}{a_0} \; \sum_{i=1}^{p(\ell)}  J^{(\ell)}_i ! \, \prod^{\ell}_{r=1} \frac{1}{m_{ri} !} \left[ - \frac{a_r}{a_0} \right]^{m_{ri}} .
\end{align}
We substitute the full expression of equation \eqref{1/h(lambda)} into equation \eqref{theorem2.1-005} and get
\begin{align}
    \label{reg-lim-1}
    \begin{split}
    \lim^\times_{\lambda \to \lambda_o} \frac{f(\lambda)}{g(\lambda)} =  & \frac{1}{n!} \left[ \frac{f^{(n)}(\lambda_0)}{h(\lambda_0)} +  \sum_{k=1}^n \binom{n}{k} f^{(n-k)}(\lambda_0)  \right. \\
    & \left. \times \; \frac{d^k}{d\lambda^k} \left[ \sum_{\ell=0}^\infty \; \frac{(\lambda - \lambda_0)^\ell}{a_0} \; \sum_{i=1}^{p(\ell)}  J^{(\ell)}_i ! \, \prod^{\ell}_{r=1} \frac{1}{m_{ri} !} \left[ - \frac{a_r}{a_0} \right]^{m_{ri}}  \right]_{\lambda = \lambda_0} \right].
    \end{split}
\end{align}
We perform the \(k^{\text{th}}\) derivative in equation \eqref{reg-lim-1} and simplify the result to obtain
\begin{align}
    \label{reg-lim-2}
    \begin{split}
    \lim^\times_{\lambda \to \lambda_o} \frac{f(\lambda)}{g(\lambda)} =  & \frac{1}{n!} \left[ \frac{f^{(n)}(\lambda_0)}{h(\lambda_0)} +  \frac{1}{a_0} \sum_{k=1}^n \binom{n}{k} f^{(n-k)}(\lambda_0)  \right. \\
    & \left. \times \; \left[ \sum_{\ell=k}^\infty \; \frac{\ell !}{(\ell - k)!} \; (\lambda - \lambda_0)^{\ell-k} \; \sum_{i=1}^{p(\ell)}  J^{(\ell)}_i ! \, \prod^{\ell}_{r=1} \frac{1}{m_{ri} !} \left[ - \frac{a_r}{a_0} \right]^{m_{ri}}  \right]_{\lambda = \lambda_0} \right].
    \end{split}
\end{align}
The series represents a convergent power series expansion of the reciprocal of an analytic function \( h(\lambda) \), with \( h(\lambda_0) \neq 0 \). Therefore, \( 1/h(\lambda) \) is analytic in a neighborhood of \( \lambda_0 \), and term-by-term differentiation is valid. 

When we evaluate equation \eqref{reg-lim-2} at \(\lambda = \lambda_0\), only the term with \(\ell = k\) survives, reducing the expression to
\begin{align}
    \label{reg-lim-3}
    \lim^\times_{\lambda \to \lambda_o} \frac{f(\lambda)}{g(\lambda)} =  \frac{1}{n!} \left[ \frac{f^{(n)}(\lambda_0)}{h(\lambda_0)} +  \frac{1}{a_0} \sum_{k=1}^n k! \, \binom{n}{k} f^{(n-k)}(\lambda_0) \sum_{i=1}^{p(k)}  J^{(k)}_i ! \, \prod^{k}_{r=1} \frac{1}{m_{ri} !} \left[ - \frac{a_r}{a_0} \right]^{m_{ri}} \right].
\end{align}
To simplify equation \eqref{reg-lim-3} further and eliminate unnecessary variables, we expand \(g(\lambda)\) about \(\lambda_0\) as
\begin{align}
    \label{g(lambda)}
    g(\lambda) = \sum_{\ell=0}^\infty \frac{g^{(\ell)} (\lambda_0)}{\ell!} \,(\lambda - \lambda_0)^\ell.
\end{align}
By imposing the condition that \(g(\lambda)\) has a zero of order \(n\) at \(\lambda_0\), we rewrite the series as
\begin{align}
    \label{g(lambda)-1}
    g(\lambda) = \sum_{\ell=0}^\infty \frac{g^{(\ell +n)} (\lambda_0)}{(\ell +n)!} \, (\lambda - \lambda_0)^{\ell +n}.
\end{align}
Using equations \eqref{g(lambda)-0} and \eqref{g(lambda)-1}, we identify \(h(\lambda)\) as
\begin{align}
    \label{h(lambda)-1}
    h(\lambda) = \sum_{\ell=0}^\infty \frac{g^{(\ell +n)} (\lambda_0)}{(\ell +n)!} \, (\lambda - \lambda_0)^{\ell},
\end{align}
with
\begin{align}
    \label{h(lambda)-2}
    a_\ell = \frac{g^{(\ell +n)} (\lambda_0)}{(\ell +n)!}.
\end{align}
Substituting equations \eqref{h(lambda)-1} and \eqref{h(lambda)-2} into equation \eqref{reg-lim-3} leads directly to equation \eqref{theorem2.1-001}.    
\end{proof}

Using the definitions of \(p(k)\), \(J_i^{(k)}\), and \(m_{ri}\), we can also write the second term of equation \eqref{theorem2.1-001} as a sum over integer partitions of \(k\) providing another expression given by
\begingroup
\allowdisplaybreaks
\begin{align}
\begin{split}
    \label{reg-lim-rest-sum}
    \lim^\times_{\lambda \to \lambda_o} \frac{f(\lambda)}{g(\lambda)} = & \frac{f^{(n)} (\lambda_o) }{g^{(n)} (\lambda_o) } + \frac{1}{g^{(n)} (\lambda_o)} \sum^{n}_{k=1} k! \binom{n}{k} \,  f^{(n-k)}(\lambda_o) \\
    & \hspace{1.5cm} \times \sum^{k}_{t=1} (-1)^t \sum_{\substack{r_1 + \cdots + r_t = k \\ r_i \in \mathbb{N}}} \; \prod^{t}_{p=1} \frac{n!}{(n + r_p)!} \frac{g^{(n+r_p)}( \lambda_o)}{g^{(n)}( \lambda_o)}.
\end{split}
\end{align}
\endgroup

Using Theorem~\ref{theorem2.1}, the following corollaries can be determined.

\begin{corollary}
\label{reglim-n1}
    Let \(f(\lambda)\) be anaytic at \(\lambda_0\) with \(f(\lambda_0) \neq 0\) and \(\lambda_0\) be a simple zero of \(g(\lambda)\), then
    \begin{equation}
        \label{theorem2.1-007}
        \lim^{\times}_{\lambda \to \lambda_o} \frac{f(\lambda)}{g(\lambda)} = \frac{f'(\lambda_o)}{g'(\lambda_o)}-\frac{f(\lambda_o) g''(\lambda_o) }{2 g'(\lambda_o)^2}.
    \end{equation}
\end{corollary}

\begin{corollary}
\label{reglim-n2}
    Let \(f(\lambda)\) be anaytic at \(\lambda_0\) with \(f(\lambda_0) \neq 0\) and \(\lambda_0\) be a double zero of \(g(\lambda)\), then
    \begin{equation}
        \label{theorem2.1-009}
        \lim^{\times}_{\lambda \to \lambda_o} \frac{f(\lambda)}{g(\lambda)} = \frac{f''(\lambda_o)}{g''(\lambda_o)}-\frac{2 f'(\lambda_o) g^{(3)}(\lambda_o) }{3 g''(\lambda_o)^2}+ f(\lambda_o) \left(\frac{2 g^{(3)}(\lambda_o)^2}{9 g''(\lambda_o)^3}-\frac{g^{(4)}(\lambda_o)}{6 g''(\lambda_o)^2}\right).
    \end{equation}
\end{corollary}

\begin{corollary}
\label{reglim-n3}
    Let \(f(\lambda)\) be anaytic at \(\lambda_0\) with \(f(\lambda_0) \neq 0\) and \(\lambda_0\) be a zero of order 3 of \(g(\lambda)\), then
    \begin{align}
        \label{reg-lim-n3}
        &\lim^{\times}_{\lambda \to \lambda_o} \frac{f(\lambda)}{g(\lambda)} = \frac{f^{(3)}(\lambda_o)}{g^{(3)}(\lambda_o)}-\frac{3 f''(\lambda_o) g^{(4)}(\lambda_o) }{4 g^{(3)}(\lambda_o)^2}+ f'(\lambda_o) \left(\frac{3 g^{(4)}(\lambda_o)^2}{8 g^{(3)}(\lambda_o)^3}-\frac{3 g^{(5)}(\lambda_o)}{10 g^{(3)}(\lambda_o)^2}\right) \\
        &\hspace{3cm} + f(\lambda_o) \left(-\frac{g^{(6)}(\lambda_o)}{20 g^{(3)}(\lambda_o)^2}-\frac{3 g^{(4)}(\lambda_o)^3}{32 g^{(3)}(\lambda_o)^4}+\frac{3 g^{(5)}(\lambda_o) g^{(4)}(\lambda_o)}{20 g^{(3)}(\lambda_o)^3}\right). \nonumber
    \end{align}
\end{corollary}

\begin{corollary}
\label{reglim-n4}
    Let \(f(\lambda)\) be anaytic at \(\lambda_0\) with \(f(\lambda_0) \neq 0\) and \(\lambda_0\) be a zero of order 4 of \(g(\lambda)\), then
    \begin{align}
        \label{reg-lim-n4}
        &\lim^{\times}_{\lambda \to \lambda_o} \frac{f(\lambda)}{g(\lambda)} = \frac{f^{(4)}(\lambda_o)}{g^{(4)}(\lambda_o)}-\frac{4  f^{(3)}(\lambda_o) g^{(5)}(\lambda_o) }{5 g^{(4)}(\lambda_o)^2}+ f''(\lambda_o) \left(\frac{12 g^{(5)}(\lambda_o)^2}{25 g^{(4)}(\lambda_o)^3}-\frac{2 g^{(6)}(\lambda_o)}{5 g^{(4)}(\lambda_o)^2}\right)   \\
        &\hspace{2.5cm}+ f'(\lambda_o) \left(-\frac{4 g^{(7)}(\lambda_o)}{35 g^{(4)}(\lambda_o)^2}-\frac{24 g^{(5)}(\lambda_o)^3}{125 g^{(4)}(\lambda_o)^4}+\frac{8 g^{(6)}(\lambda_o) g^{(5)}(\lambda_o)}{25 g^{(4)}(\lambda_o)^3}\right)  \nonumber \\   
        &\hspace{3cm} + f(\lambda_o) \left(-\frac{g^{(8)}(\lambda_o)}{70 g^{(4)}(\lambda_o)^2}+\frac{2 g^{(6)}(\lambda_o)^2}{75 g^{(4)}(\lambda_o)^3}+\frac{24 g^{(5)}(\lambda_o)^4}{625 g^{(4)}(\lambda_o)^5}\right.  \nonumber \\
        &\hspace{5cm} \left. +\frac{8 g^{(7)}(\lambda_o) g^{(5)}(\lambda_o)}{175 g^{(4)}(\lambda_o)^3} -\frac{12 g^{(6)}(\lambda_o) g^{(5)}(\lambda_o)^2}{125 g^{(4)}(\lambda_o)^4}\right). \nonumber
    \end{align}
\end{corollary}


\section{Contour Integral Representations of Finite-part Integrals with logarithmic Singularities}
\label{contour-integral-fpi}

It is shown in~\cite{galapon2023} that a suitable contour integral representation offers an effective means of evaluating finite-part integrals. In particular, this representation is used to compute the finite-part of Mellin-type divergent integrals in the absence of logarithmic singularities. The representation takes the form
\begin{equation}
    \label{eq3.1}
    \bbint{0}{a} \frac{k(t)}{t^\lambda} \,\mathrm{d}t = \frac{1}{(e^{-2\pi i \lambda} - 1)} \int_C \frac{k(z)}{z^\lambda} \,\mathrm{d}z,
\end{equation}
where the contour \( C \), illustrated in Figure~\ref{fig:3-001}, straddles the branch cut of \( z^{-\lambda} \) and begins and ends at the point \( a \). The contour neither encloses any pole nor crosses any branch cut of \( k(z) \). Equation \eqref{eq3.1} applies to all non-integer values of \( \lambda \) in the half-plane \( \mathrm{Re}(\lambda) \geq 1 \) and serves as the starting point of this section.

\begin{figure}[t]
    \centering
    \includegraphics[width=9cm]{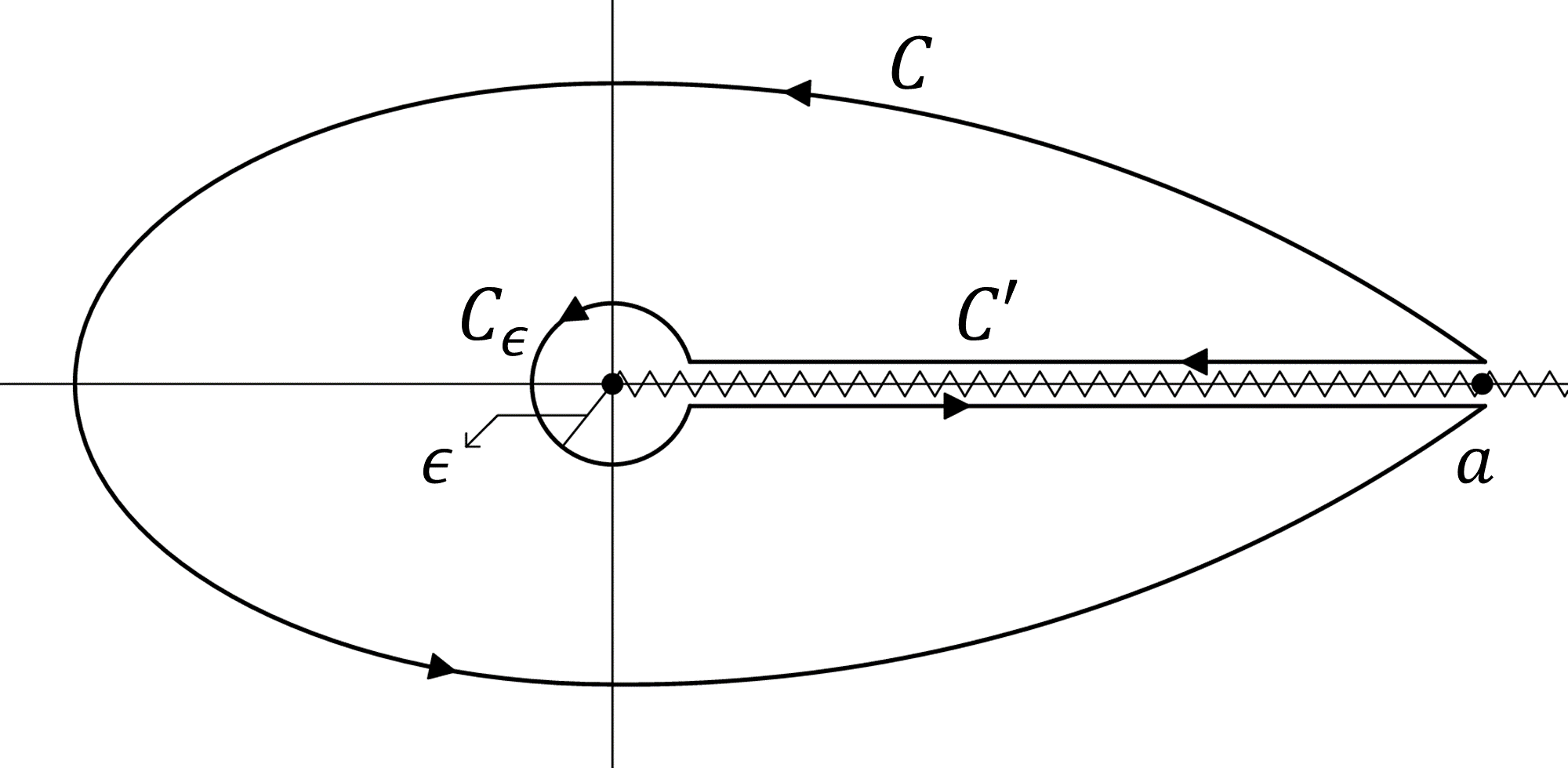}
    \caption{Contour \(C\) deformed into contour \(C'\). This contour of integration plays a pivotal role in developing the contour integral representation of the finite-part integral with logarithmic singularities. It does not enclose any pole or cross any branch cut of \(k(z)\). The branch cut for \(z^\lambda\) lies along the positive real axis.}
    \label{fig:3-001}
\end{figure}

We can extend this result to accommodate the presence of logarithmic singularities by exploiting the relationship between the finite-part integral and the Mellin transform, as established in 
\cite{galapon2023}:
\begin{equation}
    \label{eq3.2}
    \bbint{0}{a} \frac{k(t)\ln^n{t}}{t^{\lambda}} \mathrm{d}t =
        \begin{cases}
        \displaystyle \int_0^a \frac{k(t)\ln^n{t}}{t^{\lambda}} \mathrm{d}t, & d<\mathrm{Re}(\lambda)<1, \\[10pt]
        \displaystyle \mathrm{AC} \int_0^a \frac{k(t)\ln^n{t}}{t^{\lambda}} \mathrm{d}t, & \mathrm{Re}(\lambda) \geq 1,\ \lambda \notin \mathbb{Z}, \\[10pt]
        \displaystyle \lim^\times_{\lambda \to \mathfrak{b} } \mathrm{AC} \int_0^a \frac{k(t)\ln^n{t}}{t^{\lambda}} \mathrm{d}t, &  \mathrm{Re}(\lambda) \geq 1,\ \lambda \in \mathbb{N}.
        \end{cases}
\end{equation}
Here, if \( a < \infty \), then \( d = -\infty \). Conversely, if \( a = \infty \), then \( d \) may be finite or equal to negative infinity, depending on the behavior of \( k(t) \) as \( t \to \infty \). The notation AC refers to analytic continuation. The function \(k(t)\) satisfies the following conditions: \(k(0) \neq 0\); it admits a complex extension \(k(z)\) that is analytic in the interval \([0, a)\); and it possesses a Mellin transform. We impose the same conditions in this section, except for the last one. In fact, we show that one can still determine the finite-part integral even when the Mellin transform of \(k(t)\) does not exist. The next section discusses this in detail through examples.

Equation~\eqref{eq3.2} outlines three scenarios based on the values of \(\lambda\), each dictating a specific connection between the finite-part integral and the Mellin transform. The Mellin transform \(\int_0^a t^{-\lambda} \, k(t)\ln^n{t} \, \mathrm{d}t\) exists in the strip \(\mathrm{Re} (\lambda) < 1\) for a finite \(a\) or in the strip \(d < \mathrm{Re}(\lambda) < 1\) for \(a=\infty\). If we apply the definition of the finite-part integral \(\bbint{0}{a} t^{-\lambda} \, k(t)\ln^n{t} \, \mathrm{d}t\) in this strip of analyticity, the divergent part (\(D_\epsilon\)) is zero. It means that when \(d < \mathrm{Re}(\lambda) < 1\), the finite-part integral coincides with the Mellin transform. For non-integer \(\mathrm{Re}(\lambda) \geq 1\), the finite-part integral corresponds to the analytic continuation of the Mellin transform. Finally, for integer \(\mathrm{Re}(\lambda) \geq 1\), the finite-part integral equals the regularized limit of the analytic continuation as \(\lambda \to \mathfrak{b}\) where \(\mathfrak{b} \in \mathbb{N}\).

The finite-part integral given in equation \eqref{eq3.1} assumes a representation given by~\cite{galapon2023}
\begin{equation}
    \bbint{0}{a} \frac{k(t)}{t^{\lambda}} \mathrm{d}t = \int_{\epsilon}^a \frac{k(t)}{t^\lambda} \mathrm{d}t + \sum_{l=0}^\infty a_l \frac{\epsilon^{l-\lambda+1}}{l-\lambda+1}
\end{equation}
in the region \(\mathrm{Re} (\lambda) \geq 1, \; \lambda \notin \mathbb{Z}\) for \(\epsilon < \rho_o, \, a \) where \(\rho_o\) is the distance of the singularity of \(k(z)\) nearest to the origin. Given these conditions, the complex extension \(k(z)\) of \(k(t)\) can be represented by the power series \(k(z) = \sum_{l = 0}^\infty a_l z^l\). To extend this to cases involving logarithmic singularities, we invoke the identity~\cite{galapon2023} 
\begin{equation}
    \label{eq3.3}
    \bbint{0}{a} \frac{k(t)\ln^n{t}}{t^{\lambda}} \mathrm{d}t = (-1)^n \frac{d^n}{d\lambda^n} \bbint{0}{a} \frac{k(t)}{t^{\lambda}} \mathrm{d}t,
\end{equation}
which holds for any positive integer \(n\). The right-hand side integral admits differentiation under the integral sign to any order, and, for the associated infinite series, permits term-by-term differentiation to the same extent under the assumptions that \(k(t)\) is continuous on \([\epsilon, a]\) with \(0< \epsilon < a \), the integral converges for the values of \(\lambda\), and the left-hand side exists. This paper classifies the contour integral representations into two categories based on whether \(\lambda\) is an integer or a non-integer.

Before we proceed, let us clarify the definitions of the complex logarithms employed in the following sections. We define the complex logarithm as \(\ln z = \ln |z| + i\,\mathrm{arg}\, z\), where the argument satisfies \(0 \leq \mathrm{arg}\, z < 2\pi\). It coincides with the real natural logarithm \(\ln t\) along the upper side of the positive real axis. On the other hand, the complex principal value logarithm is given by \(\mathrm{Log}\, z = \ln |z| + i\,\mathrm{Arg}\, z\) where \( -\pi < \mathrm{Arg}\, z \leq \pi \).

\subsection{Contour integral representation for noninteger \texorpdfstring{\(\lambda\)}{lambda}}
\label{contour-noninteger}
\begin{theorem}
    \label{theorem_3.1}
    Let $k(z)$ be analytic in the interval $[0,a]$. Then the contour integral representation of $\bbint{0}{a} t^{-\lambda} k(t) \ln^n{t}\,\mathrm{d}t$ for all non-integer \(\lambda\) in the half-plane \(\mathrm{Re}(\lambda) \geq 1\) is given by
    \begin{align}
    \label{theorem3.1}
        \begin{split}
        \bbint{0}{a} \frac{k(t) \ln^n t}{t^{\lambda}}\,\mathrm{d}t  = \sum^{n}_{j = 0} \binom{n}{j} (2 \pi i)^{n-j}  \, \beta_j (\lambda) \int_C \frac{k(z)}{z^\lambda} \ln^{j}{z} \mathrm{d}z,
        \end{split}
    \end{align}
    and \(\beta_j (\lambda)\) is given by
    \begin{align}
        \beta_j (\lambda) = \sum^{n-j}_{l =0} (-1)^l \; l ! \;\mathbf{S}^{n-j}_l \; \frac{\; e^{-2\pi i \lambda l}}{(e^{-2\pi i \lambda } - 1)^{l +1}}
    \end{align}
    where n is any positive integer and \(\mathbf{S}^{n-j}_l\) is the Stirling number of the second kind.
\end{theorem}

\begin{proof}
To prove this theorem, begin with the initial condition given in equation \eqref{eq3.1}, where there is no logarithmic singularity. Apply equation \eqref{eq3.3} with \(n=1\) to equation \eqref{eq3.1} to derive the contour integral representation of \(\bbint{0}{a} t^{-\lambda} k(t) \ln t\,\mathrm{d}t\), which is given by
\begin{align}
    \label{eq3.6}
    \bbint{0}{a} \frac{k(t) \ln{t}}{t^{\lambda}} \mathrm{d}t &= \int_C \frac{k(z) }{z^{\lambda } } \left[ \frac{\ln{z}}{(e^{-2 i \pi  \lambda }-1)} - \frac{2 i \pi  e^{-2 i \pi  \lambda }}{(e^{-2 i \pi  \lambda }-1)^2} \right]  \mathrm{d}z.
\end{align}
By repeatedly applying equation \eqref{eq3.3} to equation \eqref{eq3.1}, we obtain the desired contour integral representation of \(\bbint{0}{a} t^{-\lambda} k(t) \ln^n t\,\mathrm{d}t\). The case when \(n=2\) is expressed as
\begingroup
\allowdisplaybreaks
\begin{align}
    \label{eq3.7}
    \bbint{0}{a} \frac{k(t) \ln^2{t}}{t^{\lambda}} \mathrm{d}t &= \int_C \frac{k(z) }{z^{\lambda } } \left[ \frac{\ln^2{z}}{(e^{-2 i \pi  \lambda }-1)} - \frac{4 i \pi  e^{-2 i \pi  \lambda }\ln{z}}{(e^{-2 i \pi  \lambda }-1)^2}  \right. \\
    &\hspace{5cm} \left. - \frac{4 \pi ^2 e^{-2 i \pi  \lambda } \left(e^{-2 i \pi  \lambda }+1\right) }{\left(e^{-2 i \pi  \lambda } - 1\right)^3} \right]  \mathrm{d}z, \nonumber
\end{align}
\endgroup
We can obtain other cases of contour integral representation by the same method outlined above.

Observing equations \eqref{eq3.6}, \eqref{eq3.7}, and other cases with higher-order \(n\), we can conclude that the contour integral representation of \(\bbint{0}{\infty} t^{-\lambda} k(t) \ln^n{t} \, \mathrm{d}t\) satisfies the relation
\begin{equation}
    \label{eq3.11}
    \bbint{0}{a} \frac{k(t)\ln^n t}{t^{\lambda}}\,\mathrm{d}t = \sum_{l=0}^{n} \frac{1}{\left(e^{-2\pi i \lambda }-1\right)^{l+1}} \int_C \frac{k(z)}{z^{\lambda}} \, G^n_l(z)\, \mathrm{d}z
\end{equation}
where the functions \(G^n_l(z)\) must be determined. 

Let \(n=0\) in equation \eqref{eq3.11} and compare it to equation \eqref{eq3.1}; the initial condition is given by
\begin{equation}
    \label{eq3.12}
    G^0_0(z) = 1.
\end{equation}
To begin solving for \(G^n_l(z)\), consider the contour integral given by
\begin{equation}
    \int_C  \frac{k(z) \ln^n z}{z^{\lambda}}\,\mathrm{d}z,
\end{equation}
where the contour \(C\) is the same as the one shown in figure \ref{fig:3-001}. Deform the contour of integration from \(C\) to \(C'\), as shown in figure \ref{fig:3-001}, to obtain the equality
\begin{equation}
    \label{eq3.13}
    \int_C \frac{k(z) \ln^n z}{z^{\lambda}}\,\mathrm{d}z = \int_a^{\epsilon} \frac{k(t) \ln^n t}{t^{\lambda}}\,\mathrm{d}t + \int_{\epsilon}^a \frac{k(t) (\ln t + 2\pi i)^n}{t^{\lambda} e^{2\pi \lambda i}}\,\mathrm{d}t + \int_{C_{\epsilon}} \frac{k(z)\ln^n z}{z^{\lambda}}\,\mathrm{d}z .
\end{equation}
By applying the binomial theorem to the second term, equation \eqref{eq3.13} becomes
\begin{align}
    \label{eq3.14}
    \begin{split}
    \int_C\frac{k(z) \ln^n z}{z^{\lambda}}\,\mathrm{d}z &=  (e^{-2\pi\lambda i}-1) \int_\epsilon^a \frac{k(t)\ln^n t}{t^{\lambda}}  \,\mathrm{d}t \\
    &\quad + e^{-2\pi\lambda i} \sum_{j=0}^{n-1}\binom{n}{j} (2\pi i)^{n-j}  \int_\epsilon^a \frac{k(t) \ln^j t}{t^{\lambda}}\,\mathrm{d}t + \int_{C_{\epsilon}} \frac{k(z)\ln^n z}{z^{\lambda}}\,\mathrm{d}z .
    \end{split}
\end{align}

The integral along \(C_{\epsilon}\) vanishes as \(\epsilon \to 0\) for \(\mathrm{Re}(\lambda) < 1\), while it diverges for \(\mathrm{Re}(\lambda) \geq 1\). To explain this further, consider the third term of equation \eqref{eq3.14} for \(n = 0\). For sufficiently small \(\epsilon\), \(k(z)\) remains analytic on a region containing \(C_\epsilon\) about the origin. Then we can use the power series expansion of the function \(k(z) = \sum_{l = 0}^\infty a_l z^l\) and substitute the parametrization \(z = \epsilon e^{i \theta}\), where \(0 \leq \theta \leq 2\pi\), to get
\begin{equation}
    \label{C-epsilon}
    \int_{C_{\epsilon}} \frac{k(z)}{z^{\lambda}}\,\mathrm{d}z = \int_{C_{\epsilon}} \sum_{l = 0}^\infty a_l z^{l - \lambda} \, \mathrm{d}z = i \sum_{l = 0}^\infty a_l \epsilon^{l - \lambda + 1} \int_0^{2\pi} e^{i(l - \lambda + 1) \theta} \, \mathrm{d}\theta,
\end{equation}
and perform term-by-term integration. Evaluating the integral in equation \eqref{C-epsilon}, we get
\begin{equation}
    \label{C-epsilon-01}
    \int_{C_{\epsilon}} \frac{k(z)}{z^{\lambda}}\,\mathrm{d}z = - (e^{-2\pi \lambda i} - 1) \sum_{l = 0}^\infty a_l \frac{\epsilon^{l - \lambda + 1}}{\lambda - l - 1}.
\end{equation}
Let \(\lambda = \lambda_R + i \lambda_I\). For \(\mathrm{Re}(\lambda) < 1\), equation \eqref{C-epsilon-01} vanishes as \(\epsilon \to 0\). For the other case, \(\mathrm{Re}(\lambda) \geq 1\), equation \eqref{C-epsilon-01} can be expressed as
\begin{equation}
    \label{C-epsilon-02}
    \int_{C_{\epsilon}} \frac{k(z)}{z^{\lambda}}\,\mathrm{d}z = - (e^{-2\pi \lambda i} - 1) \Biggl[ \sum_{l = 0}^{\lfloor \lambda_R - 1 \rfloor} a_l \frac{\epsilon^{l - \lambda + 1}}{\lambda - l - 1} + \sum_{l = \lfloor \lambda_R - 1 \rfloor + 1}^{\infty} a_l \frac{\epsilon^{l - \lambda + 1}}{\lambda - l - 1} \Biggr].
\end{equation}
As \(\epsilon \to 0\), the second term of equation \eqref{C-epsilon-02} vanishes while the first term diverges.

From \cite{galapon2023}, we know that equation~\eqref{C-epsilon-02} represents the divergent part subtracted from the divergent integral in order to obtain the finite-part. Accordingly, the finite-part integral is given by
\begin{equation}
    \label{C-epsilon-04}
    \bbint{0}{a} \frac{k(t)}{t^\lambda} \, \mathrm{d}t = \frac{1}{e^{-2\pi i \lambda} - 1} \int_C \frac{k(z)}{z^{\lambda}}\,\mathrm{d}z = \lim_{\epsilon \to 0} \left[ \int_\epsilon^a \frac{k(t)}{t^\lambda} \, \mathrm{d}t - \sum_{l = 0}^{\lfloor \lambda_R - 1 \rfloor} a_l \frac{\epsilon^{l - \lambda + 1}}{\lambda - l - 1} \right],
\end{equation}
where the coefficients \( a_l \) are the terms in the local expansion of \( k(t) \) near the origin. The existence of the left-hand side implies that the limit exists.

We can use the same method to solve the integral around \(C_{\epsilon}\) for \(n = 1\). This gives us the following representation:
\begin{align}
    \label{C-epsilon-05}
    \begin{split}
    \int_{C_{\epsilon}} \frac{k(z) \ln z}{z^{\lambda}}\,\mathrm{d}z = - (e^{-2\pi  i \lambda } - 1) &  \Biggl[ \sum_{l = 0}^{\lfloor \lambda_R - 1 \rfloor} a_l \left( \frac{\epsilon^{l - \lambda + 1} \ln \epsilon}{\lambda - l - 1} + \frac{\epsilon^{l - \lambda + 1}}{(\lambda - l - 1)^2} \right) \\
    & + \sum_{l = \lfloor \lambda_R - 1 \rfloor + 1}^{\infty} a_l \left( \frac{\epsilon^{l - \lambda + 1} \ln \epsilon}{\lambda - l - 1} + \frac{\epsilon^{l - \lambda + 1}}{(\lambda - l - 1)^2} \right)  \Biggr] \\
    -2\pi i e^{-2\pi i \lambda} & \Biggl[ \sum_{l = 0}^{\lfloor \lambda_R - 1 \rfloor} a_l \frac{\epsilon^{l - \lambda + 1}}{\lambda - l - 1} + \sum_{l = \lfloor \lambda_R - 1 \rfloor + 1}^{\infty} a_l \frac{\epsilon^{l - \lambda + 1}}{\lambda - l - 1}  \Biggr].
    \end{split}
\end{align}

Equation \eqref{C-epsilon-05} represents the divergent part of the integral \(\int_\epsilon^a t^{-\lambda} k(t) \ln t \, \mathrm{d}t\) in the limit \(\epsilon \to 0\), for \(\mathrm{Re} (\lambda) \geq 1\). In fact, we can generalize the expression of the integral along \(C_{\epsilon}\) for arbitrary order \(n\) as follows:
\begingroup
\allowdisplaybreaks
\begin{align}
    \label{C-epsilon-06}
    \int_{C_{\epsilon}} \frac{k(z) \ln^n z}{z^{\lambda}} & \, \mathrm{d}z = - (e^{-2\pi i \lambda } - 1) \Biggl[ \; \sum_{l = 0}^{\lfloor \lambda_R - 1 \rfloor} a_l (-1)^n \frac{\mathrm{d}^n}{\mathrm{d}\lambda^n}\Biggl( \frac{\epsilon^{l - \lambda + 1}}{\lambda - l - 1} \Biggr) \\
    & \hspace{4cm} + \sum_{l = \lfloor \lambda_R - 1 \rfloor + 1}^{\infty} a_l (-1)^n \frac{\mathrm{d}^n}{\mathrm{d}\lambda^n}\Biggl( \frac{\epsilon^{l - \lambda + 1}}{\lambda - l - 1} \Biggr) \Biggr] \nonumber \\
    & - e^{-2\pi i\lambda} \sum_{j=0}^{n-1} \binom{n}{j} (2\pi i)^{n-j} \Biggl[ \; \sum_{l = 0}^{\lfloor \lambda_R - 1 \rfloor} a_l (-1)^j \frac{\mathrm{d}^j}{\mathrm{d}\lambda^j} \Biggl( \frac{\epsilon^{l - \lambda + 1}}{\lambda - l - 1} \Biggr) \nonumber \\
    & \hspace{4cm} + \sum_{l = \lfloor \lambda_R - 1 \rfloor + 1}^{\infty} a_l (-1)^j \frac{\mathrm{d}^j}{\mathrm{d}\lambda^j}\Biggl( \frac{\epsilon^{l - \lambda + 1}}{\lambda - l - 1} \Biggr) \Biggr]. \nonumber
\end{align}
\endgroup
This is a direct consequence of the relations and conditions stated in equation \eqref{eq3.3}.

Substituting equation \eqref{C-epsilon-06} into equation \eqref{eq3.14}, we obtain
\begingroup
\allowdisplaybreaks
\begin{align}
    \label{C-epsilon-07}
    \begin{split}
    \int_C \frac{k(z) \ln^n z}{z^{\lambda}}\,\mathrm{d}z =  (e^{-2\pi\lambda i}-1) & \Biggl[ \int_\epsilon^a \frac{k(t)\ln^n t}{t^{\lambda}}\,\mathrm{d}t - \sum_{l = 0}^{\lfloor \lambda_R - 1 \rfloor} a_l (-1)^n \frac{\mathrm{d}^n}{\mathrm{d}\lambda^n}\Biggl( \frac{\epsilon^{l - \lambda + 1}}{\lambda - l - 1} \Biggr) \\
    & \hspace{1.5cm} - \sum_{l = \lfloor \lambda_R - 1 \rfloor + 1}^{\infty} a_l (-1)^n \frac{\mathrm{d}^n}{\mathrm{d}\lambda^n}\Biggl( \frac{\epsilon^{l - \lambda + 1}}{\lambda - l - 1} \Biggr)  \Biggr] \\
    + e^{-2\pi\lambda i} \sum_{j=0}^{n-1} \binom{n}{j} (2\pi i)^{n-j} &  \Biggl[ \int_\epsilon^a \frac{k(t) \ln^j t}{t^{\lambda}}\,\mathrm{d}t - \sum_{l = 0}^{\lfloor \lambda_R - 1 \rfloor} a_l (-1)^j \frac{\mathrm{d}^j}{\mathrm{d}\lambda^j} \left( \frac{\epsilon^{l - \lambda + 1}}{\lambda - l - 1} \right)\\
    & \hspace{1.5cm} - \sum_{l = \lfloor \lambda_R - 1 \rfloor + 1}^{\infty} a_l (-1)^j \frac{\mathrm{d}^j}{\mathrm{d}\lambda^j}\Biggl( \frac{\epsilon^{l - \lambda + 1}}{\lambda - l - 1} \Biggr)  \Biggr].
    \end{split}
\end{align}
\endgroup
The terms inside the brackets represent the finite-part of the integral \(\int_\epsilon^a t^{-\lambda} k(t) \ln^j t \, \mathrm{d}t\) in the limit \(\epsilon \to 0\) for \(\mathrm{Re}(\lambda) \geq 1\); however when \(\mathrm{Re}(\lambda) < 1\), the integral converges and the summation vanishes. Since we are allowed to take small values of \(\epsilon\) without losing the equality, equation \eqref{eq3.14} yields two final forms:
\begin{align}
    \label{C-epsilon-08}
    \begin{split}
    \int_C \frac{k(z) \ln^n z}{z^{\lambda}}\,\mathrm{d}z = (e^{-2\pi\lambda i}-1) & \int_0^a \frac{k(t)\ln^n t}{t^{\lambda}}\,\mathrm{d}t \\
    + e^{-2\pi\lambda i} \sum_{j=0}^{n-1} \binom{n}{j} & (2\pi i)^{n-j} \int_0^a \frac{k(t) \ln^j t}{t^{\lambda}}\,\mathrm{d}t, \quad \text{for } \mathrm{Re}(\lambda) < 1,
    \end{split}
\end{align}
and
\begin{align}
    \label{C-epsilon-09}
    \begin{split}
    \int_C \frac{k(z) \ln^n z}{z^{\lambda}}\,\mathrm{d}z = (e^{-2\pi\lambda i}-1) & \bbint{0}{a} \frac{k(t)\ln^n t}{t^{\lambda}}\,\mathrm{d}t \\
    + e^{-2\pi\lambda i} \sum_{j=0}^{n-1} \binom{n}{j} & (2\pi i)^{n-j} \bbint{0}{a} \frac{k(t) \ln^j t}{t^{\lambda}}\,\mathrm{d}t, \quad \text{for } \mathrm{Re}(\lambda) \geq 1.
    \end{split}
\end{align}
We can substitute equation \eqref{eq3.11} into the second term of equations \eqref{C-epsilon-08} and \eqref{C-epsilon-09}, but within a different context as determined by the relation in equation \eqref{eq3.2}, and subsequently express the finite-part integral in terms of contour integrals. 

After performing the substitution, interchanging the summation properly, and adjusting the indices to match equation \eqref{eq3.11}, we arrive at the form
\begin{align}
    \label{eq3.16}
    \begin{split}
    \bbint{0}{a} \frac{k(t)\ln^n t}{t^{\lambda}}\,\mathrm{d}t = \frac{1}{e^{-2\pi\lambda i}-1} & \int_C \frac{k(z) \ln^n z}{z^{\lambda}}\,\mathrm{d}z \\
    - \sum_{l=1}^n \frac{1}{(e^{-2\pi\lambda i}-1)^{l+1}} & \int_C \frac{k(z)}{z^{\lambda}} \left[ e^{-2\pi\lambda i} \sum_{j=l-1}^{n-1} \binom{n}{j} (2\pi i)^{n-j} G_{l-1}^j(z) \right]\,\mathrm{d}z.
    \end{split}
\end{align}
By comparing equation \eqref{eq3.11} with equation \eqref{eq3.16}, we observe that the function \(G^n_l(z)\) satisfies the recurrence relation
\begin{equation}
    \label{eq3.18}
    G^n_l(z) = - e^{-2\pi\lambda i} \sum_{j=l-1}^{n-1} \binom{n}{j} (2\pi i)^{n-j} G^j_{l-1}(z), \quad l = 1, \dots, n,
\end{equation}
with 
\begin{equation}
    \label{eq3.17}
    G^n_0(z) = \ln^n z.
\end{equation}

Using the recurrence relation, we compute \(G_1^n(z)\) as
\begin{equation}
    \label{eq3.19}
    G_1^n(z) = - e^{-2\pi\lambda i} \sum_{j=0}^{n-1} \binom{n}{j} (2\pi i)^{n-j} \ln^j z, \quad n = 1, 2, \dots
\end{equation}
To derive \(G_2^n(z)\), we first express it in terms of \(G_1^n(z)\), and since \(G_1^n(z)\) can be written using \(G_0^n(z)\), the resulting expression involves a double summation:
\begingroup
\allowdisplaybreaks
\begin{align}
    \label{eq3.20}
    G_2^n(z) &= e^{-4\pi\lambda i} \sum_{j_1=1}^{n-1} \binom{n}{j_1} (2\pi i)^{n-j_1} \sum_{j_2=0}^{j_1-1} \binom{j_1}{j_2} (2\pi i)^{j_1-j_2} \ln^{j_2} z \\
    &= e^{-4\pi\lambda i} \sum_{j=0}^{n-2} \left[ \sum_{j_1=0}^{n-j-2} \binom{n}{j_1 + j + 1} \binom{j_1 + j + 1}{j} \right] (2\pi i)^{n-j} \ln^j z \nonumber \\
    &= e^{-4\pi\lambda i} \sum_{j=0}^{n-2} \left[ \binom{n}{j} \sum_{j_1=0}^{n-j-2} \frac{(n-j)!}{(n-j-j_1-1)! \, (j_1 + 1)!} \right] (2\pi i)^{n-j} \ln^j z \nonumber \\
    &= e^{-4\pi\lambda i} \sum_{j=0}^{n-2} \left[ \binom{n}{j} \sum_{\substack{r_1 + r_2 = n-j \\ r_i \in \mathbb{N}}} \frac{(n-j)!}{r_1! r_2!} \right] (2\pi i)^{n-j} \ln^j z \nonumber
\end{align}
\endgroup
The inner sum is proportional to the generalized power series representation of the Stirling number of the second kind, given by \cite{wolframStirlingS2}
\begin{align}
    \label{S2-gen-pow-series}
    \mathbf{S}^n_m = \frac{n!}{m!} \sum_{\substack{r_1 + \cdots + r_m = n \\ r_i \in \mathbb{N}}} \frac{1}{\prod_{j=1}^m r_j!}.
\end{align}
Substituting equation \eqref{S2-gen-pow-series} into equation \eqref{eq3.20}, we obtain
\begin{align}
    G_2^n(z) = e^{-4\pi\lambda i} \sum_{j=0}^{n-2} 2! \, \mathbf{S}^{n-j}_2 \binom{n}{j} (2\pi i)^{n-j} \ln^j z.
\end{align}

We now generalize this result for an arbitrary \(l\). Using the recurrence relation repeatedly, we derive
\begin{align}
    \label{gen-G(n,l)}
    G^n_l(z) &= \left( - e^{-2\pi i \lambda} \right)^l \sum_{j_1 = l-1}^{n-1} \binom{n}{j_1} \prod_{p=1}^{l-1} \sum_{j_{p+1} = l-p-1}^{j_p - 1} \binom{j_p}{j_{p+1}} (2\pi i)^{n-j_l} \ln^{j_l} z \\
    &= \left( - e^{-2\pi i \lambda} \right)^l \sum_{j=0}^{n-l} \left[ \binom{n}{j} \sum_{\substack{r_1 + \cdots + r_l = n-j \\ r_i \in \mathbb{N}}} \frac{(n-j)!}{\prod_{t=1}^l r_t!} \right] (2\pi i)^{n-j} \ln^j z \nonumber
\end{align}
Finally, using the definition of Stirling numbers of the second kind, we express \(G_l^n(z)\) as
\begin{equation}
    \label{eq3.23}
    G_l^n(z) = \left( - e^{-2\pi i \lambda} \right)^l \sum_{j=0}^{n-l} l! \, \mathbf{S}^{n-j}_l \binom{n}{j} (2\pi i)^{n-j} \ln^j z,
\end{equation}
where \(\mathbf{S}^{n-j}_l\) denotes the Stirling number of the second kind.

We substitute equation \eqref{eq3.23} into equation \eqref{eq3.11} and interchange the order of summation, making the appropriate index adjustments, to derive the final closed-form expression given in equation \eqref{theorem3.1}.
\end{proof}

By comparing equation \eqref{theorem3.1} with equations \eqref{eq3.3} and \eqref{eq3.1}, and simplifying the terms, we arrive at the following theorem.
\begin{theorem}
    \label{nth-der}
    Let
        \begin{align}
        \label{f-n-lambda}
            f(\lambda) = \frac{1}{(e^{-2 \pi i \lambda} -1) \, z^\lambda}, \;\;\;\;\;\; \lambda \in \mathbb{R}_{> 0} \setminus \mathbb{Z}, \;\; z \in \mathbb{C} \setminus \{0\},\;\; 0< \mathrm{Arg} \, z \leq 2\pi.
        \end{align}
    Then, its derivative satisfies the relation  
        \begin{align}
        \label{f-n-lambda-01}
            f^{(n)}(\lambda) = (-1)^n  \sum^{n}_{j = 0} \binom{n}{j} (2 \pi i)^{n-j} \ln^j z \, \sum^{n-j}_{l =0} \; l ! \;\mathbf{S}^{n-j}_l \; \frac{1}{(e^{2\pi i \lambda } - 1)^l} \, f(\lambda).
        \end{align}
\end{theorem}

\subsection{Contour integral representation for integer \texorpdfstring{\(\lambda\)}{lambda}} \label{cont-intgr-lamb}

\begin{lemma} 
    \label{theorem_S2-zero}
    For all \(n \in \mathbb{N}\) and \(q \in \mathbb{N}_0\),
    \begin{align}
        \label{stirling-zero-01}
        \sum_{l=1}^{n} (-1)^{l-1} (q+l)! \sum_{k=0}^{l-1} \frac{B^{(q+l+1)}_k}{k!} \frac{(q+l)^{l-k-1}}{(l-k-1)!} \mathbf{S}^{q+n}_{q+l} = (q+1)! \, \delta_{n,1},
    \end{align}
    where \(\mathbf{S}^{q+n}_{q+l}\) denotes the Stirling number of the second kind and \(B^{(q+l+1)}_k\) denotes the Bernoulli number of order \((q+l+1)\).
\end{lemma}

\begin{proof}
The higher-order Bernoulli numbers are generally defined by their exponential generating function \cite{carlitz1960}:
\begin{align}
    \label{bernumbers-high-ord}
    \left( \frac{z}{e^z - 1} \right)^m = \sum_{n=0}^\infty B^{(m)}_n \, \frac{z^n}{n!}.
\end{align}
They also relate to the Stirling numbers of the first kind \( \mathbf{s}^{n}_{n-k} \) and the second kind \( \mathbf{S}^{k+n}_n \) through the identities:
\begin{align}
    \label{StirS1}
    \mathbf{s}^{n}_{n-k} &= \binom{n-1}{k} B^n_k, \\
    \label{StirS2}
    \mathbf{S}^{k+n}_n &= \binom{k+n}{k} B^{(-n)}_k.
\end{align}

To prove the expression given in equation \eqref{stirling-zero-01}, we employ the following properties of higher-order Bernoulli polynomials and Stirling numbers of the first and second kind \cite{jeong2005explicit, Howard01081994, Young01122017}, \cite[26.8.39]{NIST}:
\begin{gather}
    \label{rel-001}
    B^{(m)}_n (x)= \sum_{k=0}^n \binom{n}{k} B^{(m)}_k x^{n-k}, \\
    \label{rel-002}
    B^{(k+1)}_n (x) = \left( 1- \frac{n}{k} \right) B^{(k)}_n (x) + (x-k) \frac{n}{k} B^{(k)}_{n-1} (x), \\
    \label{rel-003}
    B^{(k)}_n (k-a) = (-1)^n B^{(k)}_n (a), \\
    \label{rel-004}
    \sum_{l=j}^{k} \mathbf{s}^l_j \mathbf{S}^k_l = \delta_{j,k}.
\end{gather}

Consider the summand of the outer sum in equation \eqref{stirling-zero-01}, excluding the Stirling number \(\mathbf{S}^{q+n}_{q+l}\). Using equation \eqref{rel-001}, we rewrite it as
\begin{align}
    \label{s1001}
    (-1)^{l-1} (q+1)! \binom{q+l}{l-1} B^{(q+l+1)}_{l-1} (q+l).
\end{align}
Applying relation \eqref{rel-002} and the reflection formula \eqref{rel-003} to equation \eqref{s1001} yields
\begin{align}
    \label{s1002}
    (-1)^{l-1} (q+1)! \binom{q+l-1}{l-1} B^{(q+l)}_{l-1} (q+l)
    &= (q+1)! \binom{q+l-1}{l-1} B^{(q+l)}_{l-1} \\
    &= (q+1)! \, \mathbf{s}^{q+l}_{q+1}, \nonumber
\end{align}
where we used equation \eqref{StirS1} in the last step. Substituting this expression back into equation \eqref{stirling-zero-01}, adjusting the summation index, and applying equation \eqref{rel-004}, we obtain
\begin{align}
    \label{stirling-zero-02}
    (q+l)! \sum_{l=q+1}^{q+n} \mathbf{s}^{l}_{q+1} \mathbf{S}^{q+n}_l = (q+1)! \, \delta_{n,1}.
\end{align}
This completes the proof.
\end{proof}

\begin{lemma} \label{th-ber-seckind}
    For all \(l \in \mathbb{N}\),
    \begin{align}
    \label{ber-seckind}
        b_{l+1} = (-1)^{l+1} \sum_{k=0}^{l+1} \frac{B^{(l+1)}_k}{k!} \frac{l^{\,l+1-k}}{(l+1-k)!},  
    \end{align}
    where \(b_{l+1}\) denotes the Bernoulli number of the second kind, \(b_0 =1, \; b_1 = 1/2\), and \(B^{(l+1)}_k\) denotes the Bernoulli number of order \((l+1)\).
\end{lemma}

\begin{proof}
The Bernoulli numbers of the second kind \(b_k\) are defined by the generating function \cite[24.16.5]{NIST}
\begin{align}
    \frac{t}{\ln(1+t)} = \sum_{k=0}^\infty b_k \, t^k.
\end{align}
We use the following properties \cite{Howard01081994}, \cite[24.16.6]{NIST}:
\begingroup
\allowdisplaybreaks
\begin{gather}
    \label{ber-rel-000}
    B^{(k)}_n (z+1) - B^{(k)}_n (z) = n B^{(k-1)}_{n-1} (z), \\
    \label{ber-rel-001}
    B^{(l)}_{l+k} \left( z + \frac{l}{2} \right) = (-1)^{l+k} B^{(l)}_{l+k} \left( -z + \frac{l}{2} \right), \\
    \label{ber-rel-0011}
    B^{(n-k)}_n = \frac{k-n}{k} B^{(n-k+1)}_n + \frac{(k-n) n}{k} B^{(n-k)}_{n-1}, \\
    \label{ber-rel-002}
    b_{l+1} = - \frac{B^{(l)}_{l+1}}{(l+1)! \, l}, \;\;\; b_0 = 1, \;\;\; b_1 = \frac{1}{2},
\end{gather}
\endgroup
Using equations \eqref{rel-001}, and \eqref{ber-rel-000} to \eqref{ber-rel-002}, rewrite the right-hand side of equation \eqref{ber-seckind} as
\begin{align}
    \label{ber-seckind-001}
    \mathrm{RHS} 
    &= \frac{(-1)^{l+1}}{(l+1)!} \sum_{k=0}^{l+1} \binom{l+1}{k} B^{(l+1)}_k l^{\,l+1-k} 
    = \frac{(-1)^{l+1}}{(l+1)!} B^{(l+1)}_{l+1} (l) \\
    &= \frac{1}{(l+1)!} \left( B^{(l+1)}_{l+1} + (l+1) B^{(l)}_l \right) 
    = - \frac{B^{(l)}_{l+1}}{(l+1)! \, l} = b_{l+1}. \nonumber
\end{align}
This proves the lemma.
\end{proof}

\begin{lemma} 
    \label{bernoulli-numbers}
    For all \(m \in \mathbb{N}\),
    \begin{align}
        \label{bernoulli-numbers-01}
        \frac{B_{m+1}}{m+1} = \sum_{l=1}^m (-1)^l l! \, \mathbf{S}^m_l \sum_{k=0}^{l+1} \frac{B^{(l+1)}_k}{k!} \frac{l^{l+1-k}}{(l+1-k)!},
    \end{align}
    where \(B_{m+1}\) is the Bernoulli number, \(B^{(l+1)}_k\) is the Bernoulli number of order \((l+1)\), and \(\mathbf{S}^{m}_{l}\) denotes the Stirling number of the second kind.
\end{lemma}

\begin{proof}
Using Lemma~\ref{th-ber-seckind}, the right-hand side of equation \eqref{bernoulli-numbers-01} simplifies to
\begin{align}
    \label{bernoulli-numbers-02}
    \mathrm{RHS} = - \sum_{l=1}^m l! \, \mathbf{S}^m_l b_{l+1}.
\end{align}
Equation \eqref{bernoulli-numbers-02} evaluates to \(B_{m+1} / m+1\), as given in \cite{farhi2025formulas}.
\end{proof}

\begin{theorem} 
    \label{theorem_3.2}
    Let $k(z)$ be analytic in the interval $[0,a]$. Then the contour integral representation of $\bbint{0}{a} t^{-\mathfrak{b}} k(t) \ln{}^n{t}\,\mathrm{d}t$ for all positive integer \(\mathfrak{b}\)  is given by
    \begin{flalign}
        \label{theorem3.2}
        \bbint{0}{a} \frac{k(t) \ln{}^n{t}}{t^{\mathfrak{b}}} \mathrm{d}t = & \frac{1}{2\pi i \, (n+1)} \int_C \frac{k(z)}{z^{\mathfrak{b}}} \ln{}^{n+1}{z} \, \mathrm{d}z \\
        & \hspace{0.5cm} + \sum_{j=0}^n \binom{n}{j} (2\pi i)^{n-j} \; \frac{B_{n-j+1}}{n-j+1}   \int_C \frac{k(z)}{z^{\mathfrak{b}}} \ln{}^{j}{z} \, \mathrm{d}z, \nonumber
    \end{flalign}
    where \(B_{n-j+1}\) is the Bernoulli number and n is any positive integer.
\end{theorem}

\begin{proof}
We can prove this theorem by applying the identity in equation \eqref{eq3.2} for a natural number \(\lambda\). The finite-part integral equals the regularized limit of the analytic continuation of the Mellin transform. In solving the regularized limit of equation \eqref{theorem3.1}, denoted by
\begin{align}
    \label{lambda-m}
    \lim^\times_{\lambda \to \mathfrak{b}} \left[ \, \bbint{0}{a} \frac{k(t) \ln^n t}{t^{\lambda}}\,\mathrm{d}t \, \right] &= \lim^\times_{\lambda \to \mathfrak{b}} \left[ \sum_{j = 0}^{n} \binom{n}{j} (2 \pi i)^{n-j} \sum_{l = 0}^{n-j} (-1)^l \, l! \,\mathbf{S}^{n-j}_l \right. \\
    &\hspace{2.5cm} \left. \times \frac{e^{-2\pi i \lambda l}}{(e^{-2\pi i \lambda} - 1)^{l + 1}} \int_C \frac{k(z)}{z^\lambda} \ln^j z \, \mathrm{d}z \right],\nonumber
\end{align}
we focus on the regularized limit of the inner term on the right-hand side equal to
\begin{align}
    \label{lambda-m-1}
    \lim^\times_{\lambda \to \mathfrak{b}} \frac{e^{-2\pi i \lambda l}}{(e^{-2\pi i \lambda l} - 1)^{l + 1}} \int_C \frac{k(z)}{z^\lambda} \ln^j z \, \mathrm{d}z, \qquad \mathfrak{b} \in \mathbb{N}.
\end{align}
We justify this step by invoking the linearity of the regularized limit \cite{galapon2023}. 

Define the functions
\begin{align}
    \label{lambda-m-2}
    f(\lambda) = e^{-2\pi i \lambda l} \int_C \frac{k(z)}{z^\lambda} \ln^j z \, \mathrm{d}z, \qquad g(\lambda) = \left(e^{-2\pi i \lambda l} - 1\right)^{l + 1}.
\end{align}
According to equation \eqref{theorem2.1-001}, we need to compute the following derivatives:
\begin{align}
    \label{lambda-m-3}
    f^{(l+1)}(\mathfrak{b}), \quad f^{(l+1-k)}(\mathfrak{b}), \quad g^{(l+1)}(\mathfrak{b}), \quad \text{and} \quad g^{(r+l+1)}(\mathfrak{b}).
\end{align}
We enforce that the integrand is analytic in \( z \in \mathbb{C} \), depends smoothly on the parameter \( \lambda \), and that the contour \( C \) is chosen to avoid singularities appropriately. Under these conditions, differentiation under the integral sign is justified.

To differentiate \(f(\lambda)\), we apply Leibniz’s rule:
\begin{align}
    \label{fw(m)}
    f^{(l+1)}(\mathfrak{b}) &= \frac{d^{l+1}}{d\lambda^{l+1}} \left[ e^{-2\pi i \lambda l} \int_C \frac{k(z)}{z^\lambda} \ln^j z \, \mathrm{d}z \right]_{\lambda = \mathfrak{b}} \\
    &= (-1)^{l+1} \sum_{q=0}^{l+1} \binom{l+1}{q} (2\pi i)^{l+1-q} \int_C \frac{k(z)}{z^{\mathfrak{b}}} \ln^{j+q} z \, \mathrm{d}z. \nonumber
\end{align}
We follow the same process to find \(f^{(l+1-k)}(\mathfrak{b})\):
\begin{align}
    \label{f(w-k)(m)}
    f^{(l+1-k)}(\mathfrak{b}) = (-1)^{l+1-k} \sum_{q=0}^{l+1-k} \binom{l+1-k}{q} (2\pi i)^{l+1-k-q} \int_C \frac{k(z)}{z^{\mathfrak{b}}} \ln^{j+q} z \, \mathrm{d}z.
\end{align}
Next, we compute the derivatives of \(g(\lambda)\). For \(g^{(l+1)}(\mathfrak{b})\), we binomially expand the expression and differentiate term by term:
\begin{align}
    \label{gw(m)}
    g^{(l+1)}(\mathfrak{b}) &= \frac{d^{l+1}}{d\lambda^{l+1}} \left[ \left(e^{-2\pi i \lambda l} - 1\right)^{l + 1} \right]_{\lambda = \mathfrak{b}} \\
    &= (-2\pi i)^{l+1} (l+1)!. \nonumber
\end{align}
Here we used the series representation of the Stirling number of the second kind \cite[26.8.6]{NIST}:
\begin{align}
    \mathbf{S}^n_k = \frac{1}{k!} \sum_{i=0}^k (-1)^i \binom{k}{i} (k - i)^n.
\end{align}
Applying the same method, we find:
\begin{align}
    \label{g(r+w)(m)}
    g^{(r+l+1)}(\mathfrak{b}) &= (-2\pi i)^{r+l+1} \sum_{t=0}^{l+1} \binom{l+1}{t} (-1)^t (l+1 - t)^{r+l+1} \\
    &= (-2\pi i)^{r+l+1} (l+1)! \, \mathbf{S}^{r+l+1}_{l+1}. \nonumber
\end{align}

We substitute equations \eqref{fw(m)}, \eqref{f(w-k)(m)}, \eqref{gw(m)}, and \eqref{g(r+w)(m)} into equation \eqref{theorem2.1-001} and use the result to evaluate equation \eqref{lambda-m}, obtaining
\begingroup
\allowdisplaybreaks
\begin{flalign}
    \label{lambda-m-01}
    \bbint{0}{a} \frac{k(t) \ln^n{t}}{t^{\mathfrak{b}}} \, \mathrm{d}t 
    &= \frac{1}{2\pi i} \int_C \frac{k(z)}{z^{\mathfrak{b}}} \ln^{n+1}{z} \, \mathrm{d}z 
    - \frac{1}{2} \int_C \frac{k(z)}{z^{\mathfrak{b}}} \ln^n{z} \, \mathrm{d}z \\
    &\quad \quad + \sum_{j=0}^{n-1} \binom{n}{j} (2\pi i)^{n-j} 
    \sum_{l=1}^{n-j} (-1)^l \, l! \, \mathbf{S}^{n-j}_{l} 
    \sum_{k=0}^{l+1} \frac{B_k^{(l+1)}}{k!} \nonumber \\
    &\quad \quad \quad \quad \times \sum_{q=0}^{l+1-k} \frac{l^{l-q-k+1}}{q! \, (l-q-k+1)!} 
    \frac{1}{(2\pi i)^q} \int_C \frac{k(z)}{z^{\mathfrak{b}}} \ln^{j+q}{z} \, \mathrm{d}z. \nonumber
\end{flalign}
\endgroup
We define \( B^{(l)}_0 = 1 \), and express \( B^{(l+1)}_k \) as
\begin{align}
    \label{stirling-pi-00}
    B^{(l+1)}_k = \frac{k!}{(-2\pi i)^k} \sum_{i=1}^{p(k)} J_i^{(k)}! 
    \prod_{r=1}^{k} \frac{1}{m_{ri}!} 
    \left[ -\frac{(-2\pi i)^r (l+1)!}{(r + l + 1)!} \, \mathbf{S}^{r + l + 1}_{l + 1} \right]^{m_{ri}}.
\end{align}
By applying the alternative expression for the regularized limit given in equation \eqref{reg-lim-rest-sum} and using the definitions of \( J_i^{(k)} \) and \( m_{ri} \), we rewrite equation \eqref{stirling-pi-00} as
\begingroup
\allowdisplaybreaks
\begin{align}
    \label{stirling-pi-02}
    B^{(l+1)}_k 
    = \sum_{t=1}^{k} (-1)^t 
    \sum_{\substack{r_1 + \cdots + r_t = k \\ r_i \in \mathbb{N}}} 
    \frac{\displaystyle \binom{k}{r_1, \ldots, r_t}}{\displaystyle \binom{r_1 + l + 1}{l + 1} \cdots \binom{r_t + l + 1}{l + 1}} 
    \, \mathbf{S}^{r_1 + l + 1}_{l + 1} \cdots \mathbf{S}^{r_t + l + 1}_{l + 1}.
\end{align}
\endgroup
The Bernoulli number \( B^{(l+1)}_k \) of order (\( l+1 \)), as expressed in \eqref{stirling-pi-02}, appears in the same form as in \cite{jeong2005explicit}.

We isolate the case $l=0$ in equation~\eqref{lambda-m-01} because $\mathbf{S}^n_0 = \delta_{n,0}$.  The evaluation is reflected by the first and second terms. The third term (LT) has a quadruple summation in the form 
\begin{flalign}
\label{app-lambda-m-03}
\mathbf{\hat{Q}_s} = \sum_{j=0}^{n-1} \; \sum_{l=1}^{n-j} \; \sum_{k=0}^{l+1} \; \sum_{q=0}^{l+1-k},
\end{flalign}
where the summand is given by
\begin{align}
\label{app-lambda-m-03-sum}
    \binom{n}{j} (2\pi i)^{n-j} 
    (-1)^l \, l! \, \mathbf{S}^{n-j}_{l} 
     \frac{B_k^{(l+1)}}{k!}  \frac{l^{l-q-k+1}}{q! \, (l-q-k+1)!} 
    \frac{1}{(2\pi i)^q} \int_C \frac{k(z)}{z^{\mathfrak{b}}} \ln^{j+q}{z} \, \mathrm{d}z. 
\end{align}
We combine the indices $j$ and $q$ and define
\begin{align}
j + q = n + 1 - m,
\end{align}
where $m$ is a positive integer from $0$ to $n+1$. We write $j$ in terms of $n$, $m$, and $q$ and rearrange all terms according to the powers of $\ln z$, resulting in an equivalent operator given by
\begingroup
\allowdisplaybreaks
\begin{flalign}
\label{app-lambda-m-04}
\mathbf{\hat{Q}_s} &= \Biggl[ 
\sum_{m=2-\delta_{n,1}}^{n} \sum_{l=1}^{m} \sum_{k=0}^{l} \delta_{q,1} 
+ \sum_{m=2}^{n-1} \sum_{q=2}^{n+1-m} \sum_{l=q-1}^{m+q-1} \sum_{k=0}^{l+1-q} \\
&\hspace{1cm} + \sum_{m=2}^{n+1} \sum_{l=1}^{m-1} \sum_{k=0}^{l+1} \delta_{q,0} 
+ \sum_{m=0}^{1 - \delta_{n,1}} \sum_{q=2-m}^{n+1-m} \sum_{l=q-1}^{m+q-1} \sum_{k=0}^{l+1-q} 
\Biggr] \delta_{j, n+1 - m - q} \nonumber \\
&= \mathbf{\hat{Q}_{s_1}} + \mathbf{\hat{Q}_{s_2}} + \mathbf{\hat{Q}_{s_3}} + \mathbf{\hat{Q}_{s_4}}. \nonumber
\end{flalign}
\endgroup

To compute $\mathbf{\hat{Q}_{s_1}}$ acting on the summand given in equation~\eqref{app-lambda-m-03-sum}, we adjust the indices and express $B^{(l+1)}_k$ using the Stirling numbers of the first kind. This yields
\begingroup
\allowdisplaybreaks
\begin{align}
\mathrm{LT}_1 &= \sum_{m=2 - \delta_{n,1}}^{n} \binom{n}{m} (2\pi i)^{m-1} \int_C \frac{k(z)}{z^{\mathfrak{b}}} \ln^{n+1 - m} z \, \mathrm{d}z 
\left( \sum_{l=1}^{m} (-1)^l \mathbf{S}^m_l \sum_{k=0}^{l} \mathbf{s}^{l+1}_{l+1-k} l^{l-k} \right) \\
&= \sum_{m=2 - \delta_{n,1}}^{n} \binom{n}{m} (2\pi i)^{m-1} \int_C \frac{k(z)}{z^{\mathfrak{b}}} \ln^{n+1 - m} z \, \mathrm{d}z 
\left( \sum_{l=1}^{m} (-1)^l (l-1)! \mathbf{S}^m_l \right) \nonumber \\
&= 0. \nonumber
\end{align}
\endgroup
The solution in simplifying the inner summation inside the parentheses uses the generating function of the Stirling numbers of the first kind and a well-known identity for the Stirling numbers of the second kind given by \cite{StirlingFirstKind, StirlingSecondKind}
\begin{gather}
    \label{gen-func-s1}
    \sum_{k=0}^n \mathbf{s}^n_k \, x^k = n! \binom{x}{n}, \\
    \sum_{l=1}^m (-1)^l \, (l-1)! \, \mathbf{S}^m_l = 0.
\end{gather}

Next, we evaluate $\mathbf{\hat{Q}_{s_2}}$:
\begin{align}
\mathrm{LT}_2 &= \sum_{m=1}^{n-2} \int_C \frac{k(z)}{z^{\mathfrak{b}}} \ln^{n - m} z \, \mathrm{d}z 
\sum_{q=0}^{n+1 - m} \binom{n}{m + q + 2} (2\pi i)^{m+2} \frac{(-1)^q}{(q + 2)!} \\
&\hspace{1cm} \times \left( \sum_{l=1}^{m+2} (-1)^{l-1} (q+l)! \sum_{k=0}^{l-1} \frac{B^{(q+l+1)}_k}{k!} 
\frac{(q+l)^{l-k-1}}{(l - k - 1)!} \mathbf{S}^{q+m+2}_{q+l} \right) \nonumber \\
&= 0. \nonumber
\end{align}
Lemma~\ref{theorem_S2-zero} ensures the inner summation vanishes.

We now apply $\mathbf{\hat{Q}_{s_3}}$:
\begingroup
\allowdisplaybreaks
\begin{align}
\mathrm{LT}_3 &= \sum_{m=1}^{n} \binom{n}{m} (2\pi i)^m \int_C \frac{k(z)}{z^{\mathfrak{b}}} \ln^{n - m} z \, \mathrm{d}z \\
&\hspace{1.5cm} \times \left( \sum_{l=1}^{m} (-1)^l \; l! \; \mathbf{S}^m_l 
\sum_{k=0}^{l+1} \frac{B^{(l+1)}_k}{k!} \frac{l^{l+1 - k}}{(l+1 - k)!} \right) \nonumber \\
&= \sum_{m=1}^{n} \binom{n}{m} (2\pi i)^m \; \frac{B_{m+1}}{m+1} 
\int_C \frac{k(z)}{z^{\mathfrak{b}}} \ln^{n - m} z \, \mathrm{d}z. \nonumber
\end{align}
\endgroup
The inner summation inside the parentheses is discussed in Lemma~\ref{bernoulli-numbers} and is equal to \(B_{m+1}/(m+1)\) where \(B_{m+1}\) is the Bernoulli number.

For $\mathbf{\hat{Q}_{s_4}}$, we consider two cases. In the first case, when \(m=0\), adjust the index and represent \(B^{(l+1)}_k\) in terms of the Stirling number of the first kind to obtain
\begin{align}
\label{ST_40}
\mathrm{LT}_{40} &= \left( \sum_{q=0}^{n-1} (-1)^{q+1} \binom{n}{q+1} \frac{1}{q + 2} \right) 
\frac{1}{2\pi i} \int_C \frac{k(z)}{z^{\mathfrak{b}}} \ln^{n+1} z \, \mathrm{d}z.
\end{align}
We evaluate the summation:
\begin{align}
\label{ST_40-01}
\sum_{q=0}^{n-1} (-1)^{q+1} \binom{n}{q+1} \frac{1}{q+2} &= 
\sum_{q=1}^{n} (-1)^q \binom{n}{q} \frac{1}{q+1} 
= \int_0^1 \sum_{q=1}^n (-1)^q \binom{n}{q} x^q \, \mathrm{d}x \\
&= \int_0^1 \left( (1 - x)^n - 1 \right) \mathrm{d}x = -\frac{n}{n+1}. \nonumber
\end{align}
Substituting equation \eqref{ST_40-01} into equation \eqref{ST_40}, we obtain:
\begin{align}
\label{ST_40-1}
\mathrm{LT}_{40} &= -\frac{n}{n+1} \, \frac{1}{2\pi i} 
\int_C \frac{k(z)}{z^{\mathfrak{b}}} \ln^{n+1} z \, \mathrm{d}z.
\end{align}
The second case is when \(m=1\),
\begingroup
\allowdisplaybreaks
    \begin{align}
        \mathrm{LT}_{41} & = \sum_{q=1}^{n} \binom{n}{q} \frac{(-1)^{q-1}}{q!} \int_C \frac{k(z)}{z^{\mathfrak{b}}} \ln{}^{n}{z} \, \mathrm{d}z \\
        & \hspace{1cm}\times \Biggl( \sum_{l=0}^1 (-1)^l \, (q+l-1)! \, \sum_{k=0}^{l} \, \frac{B^{(q+l)}_k}{k!} \, \frac{(q+l-1)^{l-k}}{(l-k)!} \, \mathbf{S}^{q}_{q+l-1} \Biggr) \nonumber\\
        & = \sum_{q=1}^{n} \binom{n}{q} \frac{(-1)^{q-1}}{q!} \int_C \frac{k(z)}{z^{\mathfrak{b}}} \ln{}^{n}{z} \, \mathrm{d}z \, \Biggl( -\frac{1}{2} (q-1) \, q! + \frac{1}{2} (q-1) \, q!  \Biggr) \nonumber \\
        & = 0. \nonumber
    \end{align}
\endgroup
The inner summation inside the brackets is simplified using the following properties of higher-order Bernoulli polynomials and the Stirling numbers of the second kind \cite{StirlingSecondKind, Norlund}:
\begin{gather}
     \label{extra-002}
     \mathbf{S}^n_n = 1, \;\;\; \mathbf{S}^n_{n-1} = \binom{n}{2}, \\
     \label{extra-001}
     B^{(n)}_0 = 1, \;\;\; B^{(n)}_1 = -\frac{n}{2}.
 \end{gather}

We now collect all surviving terms of the quadruple summation and write the resulting expansion as
    \begin{align}
        \label{LT-final}
        \mathrm{LT} & = \mathrm{LT}_{40} + \mathrm{LT}_3   \\
        & = -\frac{n}{(n+1)} \frac{1}{2\pi i} \int_C \frac{k(z)}{z^{\mathfrak{b}}} \ln{}^{n+1}{z} \, \mathrm{d}z \nonumber \\
        & \hspace{2cm}+ \sum_{m=1}^n \binom{n}{m} (2\pi i)^m \; \frac{B_{m+1}}{m+1} \int_C \frac{k(z)}{z^{\mathfrak{b}}} \ln{}^{n-m}{z} \, \mathrm{d}z. \nonumber
    \end{align}
Substituting this final expression into equation \eqref{lambda-m-01} and reordering the index of the summation leads to equation \eqref{theorem3.2}. 
\end{proof}

Using Theorem~\ref{theorem_3.2}, contour integral representations of finite-part integrals with logarithmic singularities can be determined.

\begin{corollary}
    Let $k(z)$ be analytic in the interval $[0,a]$. Then, the contour integral representation of $\bbint{0}{a} t^{-\mathfrak{b}} k(t) \,\mathrm{d}t$ for all positive integer \(\mathfrak{b}\) is given by
    \begin{align}
        \label{theorem3.2-000}
        \bbint{0}{a} \frac{k(t) }{t^{\mathfrak{b}}} \mathrm{d}t = \frac{1}{2 \pi i}\int_C \frac{k(z) }{z^{\mathfrak{b}} } \left( \,   \ln{}{z} - \pi i \, \right)  \mathrm{d}z.
    \end{align}
\end{corollary}

\begin{corollary}
    Let $k(z)$ be analytic in the interval $[0,a]$. Then, the contour integral representation of $\bbint{0}{a} t^{-\mathfrak{b}} k(t) \ln{t}\,\mathrm{d}t$ for all positive integer \(\mathfrak{b}\) is given by
    \begin{align}
        \label{theorem3.2-001}
        \bbint{0}{a} \frac{k(t) \ln{}{t}}{t^{\mathfrak{b}}} \mathrm{d}t = \frac{1}{2 \pi i}\int_C \frac{k(z) }{z^{\mathfrak{b}} } \left( \frac{1}{2}\ln{}^2{z} -\pi i \ln{}{z} -\frac{\pi^2}{3} \right)  \mathrm{d}z.
    \end{align}
\end{corollary}

\begin{corollary}
      Let $k(z)$ be analytic in the interval $[0,a]$. Then, the contour integral representation of $\bbint{0}{a} t^{-\mathfrak{b}} k(t) \ln{}^2{t}\,\mathrm{d}t$ for all positive integer \(\mathfrak{b}\) is given by
    \begin{align}
        \label{cont-int-rep-n=2}
        \bbint{0}{a} \frac{k(t) \ln{}^2{t}}{t^{\mathfrak{b}}} \mathrm{d}t = \frac{1}{2 \pi i}\int_C \frac{k(z) }{z^{\mathfrak{b} } } \left( \frac{1}{3}\ln{}^3{z} -\pi i \ln{}^2{z} -\frac{2 \pi^2}{3} \ln{}{z} \right)  \mathrm{d}z.
    \end{align}
\end{corollary}

We compare equations \eqref{theorem3.2} and \eqref{lambda-m} to equations \eqref{theorem3.1}, \eqref{f-n-lambda}, and \eqref{f-n-lambda-01} and obtain the following theorem.
\begin{theorem}
    \label{nth-der-reglim}
    Let
        \begin{align}
            f(\lambda) = \frac{1}{(e^{-2 \pi i \lambda} -1) \, z^\lambda}, \;\;\;\;\;\; \mathfrak{b} \in \mathbb{N}_0, \;\; z \in \mathbb{C} \setminus \{0\}, \;\; 0< \mathrm{Arg} \, z \leq 2\pi,
        \end{align}
    Then, it satisfies the following relation  
        \begin{align}
            \lim^{\times}_{\lambda \to \mathfrak{b}} f^{(n)}(\lambda) = (-1)^n  \Biggl[ \; \frac{1}{2\pi i \, (n+1)} \, \frac{\ln^{n+1} z}{z^{\mathfrak{b}}} + \sum^{n}_{j=0} \binom{n}{j} (2\pi i )^{n-j} \; \frac{B_{n-j+1}}{n-j+1} \, \frac{\ln^{j} z}{z^{\mathfrak{b}}}\; \Biggr].
        \end{align}
\end{theorem}

\subsection{Finite-part integral with logarithmic singularity}
The finite-part integral and Mellin transform are two closely related mathematical concepts. Equation \eqref{eq3.2} lists their relationships. To further the discussion, we formally define specific terms used in this section: \textit{Mellin-type integral}, \textit{Mellin-type divergent integral}, and \textit{non-Mellin-type divergent integral}. A Mellin-type integral is an integral for which the finite-part integral and Mellin transform yield the same solution; it falls under the first condition given in equation \eqref{eq3.2}. A Mellin-type divergent integral is a divergent integral whose finite-part equals the analytic continuation or regularized limit of the analytic continuation of its Mellin transform, depending on the conditions given in equation \eqref{eq3.2}. Finally, a non-Mellin-type divergent integral is a Mellin-type divergent integral whose Mellin transform does not exist. This raises the question: how do we determine the finite-part integral of a non-Mellin-type integral if its Mellin transform does not exist? The next section answers this question by providing a specific example.

\subsubsection{\texorpdfstring{\textbf{Mellin-type divergent integral}}{Mellin-type divergent integral}}

In quantum electrodynamics, the Euler-Heisenberg lagrangian describes a nonperturbative interaction of a spinor loop in a constant electromagnetic background field \cite{Fliegner1997, Gerald2002, Huet-Idrish-etal}. The renormalization procedure applied to the Lagrangian includes mass and charge renormalization. We subtract the free field contribution and the logarithmically divergent term from it \cite{Gerald2002}. The renormalized Euler-Heisenberg Lagrangian contains a divergent integral in the form
\begin{equation}
    \label{div-int-EHL}
    \int_0^{\infty} \frac{e^{-\beta x}\ln (\beta x)}{x^{\nu + n}} \, \mathrm{d}x,
\end{equation}
where \(0 < \nu < 1\) and \(n\) is a natural number. We can solve the finite-part of equation \eqref{div-int-EHL} through finite-part integration.

\subsubsection*{\texorpdfstring{Case \(\nu \neq 0\)}{Case nu neq 0}}
\textit{Example}. Simplify equation \eqref{div-int-EHL} to obtain
\begin{equation}
    \label{FPI-EHL}
    \bbint{0}{\infty} \frac{e^{-\beta x}\ln (\beta x)}{x^{\nu + n}} \, \mathrm{d}x = \ln \beta \bbint{0}{\infty} \frac{e^{-\beta x}}{x^{\nu + n}} \, \mathrm{d}x + \bbint{0}{\infty} \frac{e^{-\beta x}\ln x}{x^{\nu + n}} \, \mathrm{d}x.
\end{equation}

According to equation \eqref{eq3.2}, the finite-part integrals of the two terms in equation \eqref{FPI-EHL} are equal to their respective Mellin transforms. The Mellin transform of the first term exists and is given by
\begin{equation}
    \label{FPI-EHL-1}
    \ln \beta \bbint{0}{\infty} \frac{e^{-\beta x}}{x^{\nu + n}} \, \mathrm{d}x = \frac{\ln \beta }{\beta {}^{1-n-\nu}} \Gamma (1-n-\nu).
\end{equation}
For the second term, use equation \eqref{eq3.3} to obtain
\begin{equation}
    \label{FPI-EHL-2}
    \bbint{0}{\infty} \frac{e^{-\beta x}\ln x}{x^{\nu + n}} \, \mathrm{d}x = \frac{ (\psi (1-n-\nu )-\ln\beta ) }{\beta ^{1-n-\nu}}  \Gamma (1-n-\nu ).
\end{equation}
By summing equations \eqref{FPI-EHL-1} and \eqref{FPI-EHL-2}, we obtain the final result as
\begin{equation}
    \label{FPI-EHL-3}
    \bbint{0}{\infty} \frac{e^{-\beta x}\ln (\beta x)}{x^{\nu + n}} \, \mathrm{d}x = \frac{  \Gamma' (1-n-\nu ) }{\beta ^{1-n-\nu}}.
\end{equation}

\subsubsection*{\texorpdfstring{Case \(\nu = 0\)}{Case nu = 0}}\textit{Example}. From the relationships listed in equation \eqref{eq3.2}, we find that the finite-part integral of equation \eqref{FPI-EHL-3} at \(\nu = 0\) equals the regularized limit of the right-hand side expressed as 
    \begin{equation}
        \label{FPI-EHL-4}
        \lim_{\nu \to 0}^\times \, \bbint{0}{\infty} \frac{e^{-\beta x}\ln (\beta x)}{x^{\nu + n}} \, \mathrm{d}x = \lim_{\nu \to 0}^\times \, \biggl[-\frac{  1 }{\beta ^{1-n-\nu}} \frac{\mathrm{d}}{\mathrm{d}\nu} \left(\Gamma (1-n-\nu )\right) \biggr] .
    \end{equation}
By applying the reflection property of the gamma function, we rewrite equation \eqref{FPI-EHL-4} as
    \begin{equation}
        \label{FPI-EHL-5}
        \lim_{\nu \to 0}^\times \, \bbint{0}{\infty} \frac{e^{-\beta x}\ln (\beta x)}{x^{\nu + n}} \, \mathrm{d}x = \lim_{\nu \to 0}^\times \, \biggl[\frac{  1 }{\beta ^{1-n-\nu}} \frac{\mathrm{d}}{\mathrm{d}\nu} \left( \frac{(-1)^{n+1} \pi}{\Gamma (n+\nu)\sin (\nu \pi)} \right) \biggr] .
    \end{equation}
Then, perform the derivative to obtain
    \begin{equation}
        \label{FPI-EHL-6}
        \lim_{\nu \to 0}^\times \, \bbint{0}{\infty} \frac{e^{-\beta x}\ln (\beta x)}{x^{\nu + n}} \, \mathrm{d}x = \lim_{\nu \to 0}^\times \, \biggl[\frac{  (-1)^{n+1} \pi }{\beta ^{1-n-\nu}} \left( -\frac{\pi  \cot (\pi  \nu ) \csc (\pi  \nu )}{\Gamma (n+\nu )}-\frac{\csc (\pi  \nu ) \psi (n+\nu )}{\Gamma (n+\nu )} \right) \biggr] .
    \end{equation}

From equation \eqref{FPI-EHL-6}, we consider the regularized limit of the first term by defining
    \begin{equation}
        \label{FPI-EHL-7}
        f(\nu) = \frac{  (-1)^{n} \pi^2  }{\beta ^{1-n-\nu}} \left( \frac{  \cos (\pi  \nu ) }{\Gamma (n+\nu )} \right), \quad \mathrm{and} \quad g(\nu) = \sin^2 (\pi  \nu ).
    \end{equation}
Using Corollary~\ref{reglim-n2}, we calculate the first regularized limit as
   \begin{align}
        \label{FPI-EHL-8}
        \begin{split}
        \lim_{\nu \to 0}^\times \, \biggl[\frac{  (-1)^{n} \pi^2 }{\beta ^{1-n-\nu}} \left( \frac{  \cot (\pi  \nu ) \csc (\pi  \nu )}{\Gamma (n+\nu )} \right) \biggr]  = \frac{(-1)^n \beta^{n-1}}{\Gamma(n)} \biggl[ \frac{1}{2}\ln^2(\beta)  - & \ln(\beta)\psi(n) + \frac{1}{2}\psi(n)^2 \\
        & - \frac{1}{2}\psi^{(1)}(n) - \frac{\pi^2}{6} \biggr].
        \end{split}
    \end{align}
Next, consider the regularized limit of the second term in equation \eqref{FPI-EHL-6} by setting
    \begin{equation}
        \label{FPI-EHL-9}
        f(\nu) = \frac{  (-1)^n \pi }{\beta ^{1-n-\nu}} \left( \frac{\psi (n+\nu )}{\Gamma (n+\nu )} \right), \quad \mathrm{and} \quad g(\nu) = \sin (\pi  \nu )
    \end{equation}
and using Corollary~\ref{reglim-n1}. We find the second regularized limit as
   \begin{align}
        \label{FPI-EHL-10}
        \begin{split}
        \lim_{\nu \to 0}^\times \, \biggl[\frac{  (-1)^n \pi }{\beta ^{1-n-\nu}} \frac{\csc (\pi  \nu ) \psi (n+\nu )}{\Gamma (n+\nu )} \biggr] & = \frac{(-1)^n \beta ^{n-1}}{\Gamma (n)} \biggl[ \psi ^{(1)}(n)- \psi (n)^2 +\ln (\beta ) \psi (n) \biggr].
        \end{split}
    \end{align}
Finally, add equations \eqref{FPI-EHL-8} and \eqref{FPI-EHL-10} to determine the regularized limit of equation \eqref{FPI-EHL-3} as \(\nu \to 0\), yielding
    \begin{equation}
        \label{FPI-EHL-11}
        \lim_{\nu \to 0}^\times \, \bbint{0}{\infty} \frac{e^{-\beta x}\ln (\beta x)}{x^{\nu + n}} \, \mathrm{d}x = \frac{(-1)^{n+1} \beta^{n-1}}{\Gamma(n)} \left[ \frac{\pi^2}{6} - \frac{1}{2}\ln^2\beta + \frac{1}{2}\psi(n)^2 - \frac{1}{2}\psi^{(1)}(n) \right].
    \end{equation}

We recover the divergent integral that forms part of the Euler-Heisenberg Lagrangian by evaluating equation \eqref{FPI-EHL-11} at \(n=2\), which yields the finite-part integral
    \begin{equation}
        \label{FPI-EHL-12}
        \bbint{0}{\infty} \frac{e^{-\beta x}\ln (\beta x)}{x^{2}} \, \mathrm{d}x = 
        -\frac{1}{12} \beta  \left(-6 \ln ^2(\beta )+\pi ^2+6 (\gamma -2) \gamma +12\right).
    \end{equation}

\subsubsection{\texorpdfstring{\textbf{Non-Mellin-type divergent integral}}{Non-Mellin-type divergent integral}}
\label{non-mellin-example}
    \begin{figure}[t]
        \centering
        \includegraphics[width=6cm]{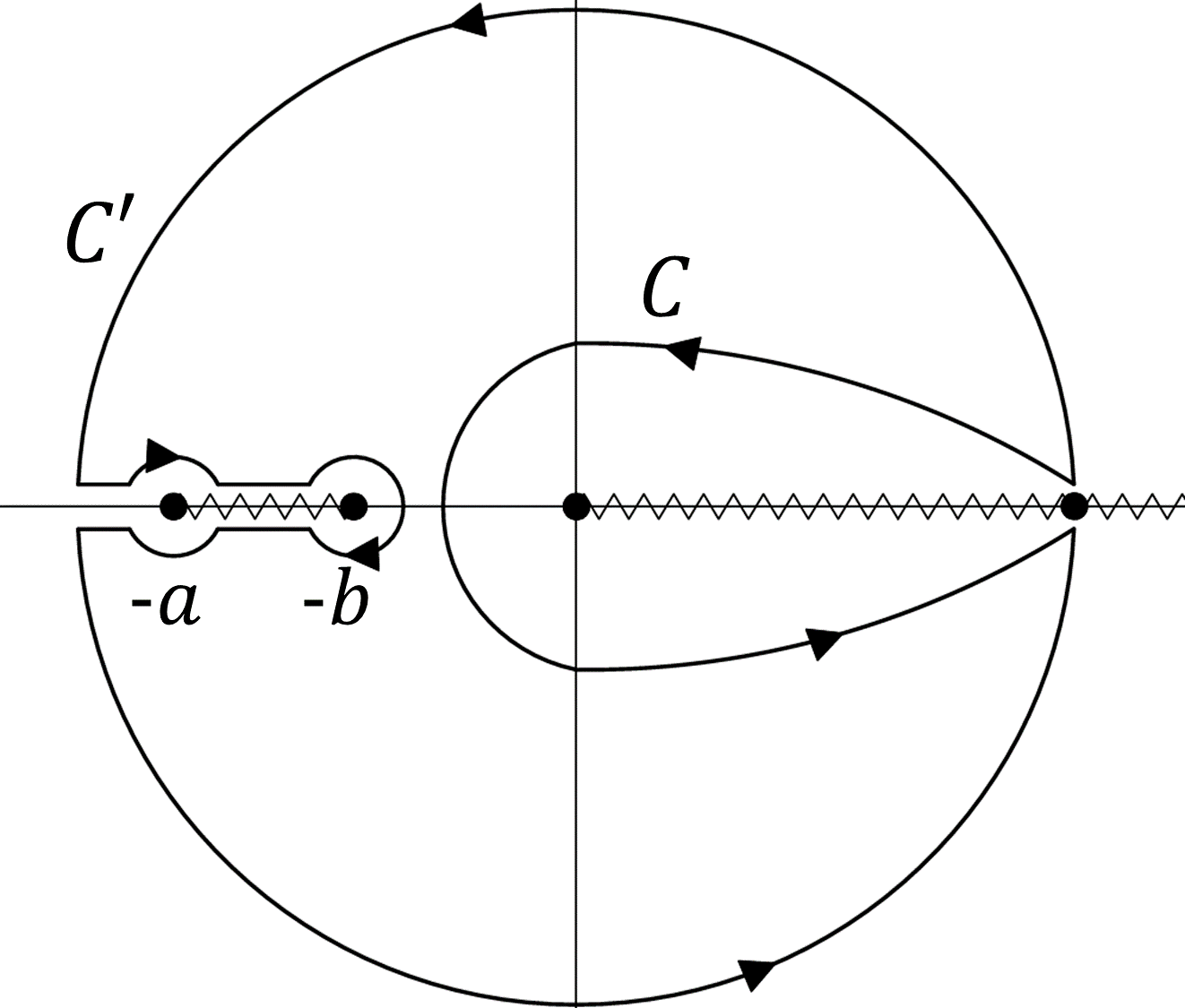}
        \caption{Contour  C deformed into C'. It does not enclose any pole or cross any branch cut of \(k(z)\). }
        \label{fig:3-002}
    \end{figure}
\noindent \textit{Example}. To address the question regarding the nonexistence of the Mellin transform for a non-Mellin-type divergent integral, consider the finite-part integral
\begin{equation}
    \label{ex4.1.1-001}
    \bbint{0}{\infty} \frac{\sqrt{a+x}}{\sqrt{b+x}} \frac{\ln{x}}{x^\lambda} \, \mathrm{d}x,
\end{equation}
where \( \mathrm{Re}(\lambda) > 1 \), \( a > b > 0 \), and \( \lambda = m \) is a positive integer. For the integrand to converge at infinity, we require \( \mathrm{Re}(\lambda) > 1 \); however, for convergence near zero, \( \mathrm{Re}(\lambda) < 1 \) is necessary. This contradiction implies that there is no strip of analyticity, that is, the Mellin transform does not exist. 

We are thus left with the task of evaluating the finite-part of this divergent integral without relying on the Mellin transform. Under the given conditions and assuming the integrand admits a complex extension, we apply equation \eqref{eq3.6} to obtain the contour integral representation
\begin{align}
    \begin{split}
    \bbint{0}{\infty} \frac{\sqrt{a+x}}{\sqrt{b+x}} \, \frac{\ln{x}}{x^{\nu + m}}\; \mathrm{d}x &= \frac{1}{ e^{-2\pi i \nu}-1 } \int_C \frac{\sqrt{a + z}}{\sqrt{b + z}} \frac{\ln z}{z^{\nu + m}} \, \mathrm{d}z \\
    &\quad - \frac{2\pi i e^{-2\pi i \nu}}{ ( e^{-2\pi i \nu}-1 )^2 } \int_C \frac{\sqrt{a + z}}{\sqrt{b + z}} \frac{1}{z^{\nu + m}} \, \mathrm{d}z. 
    \end{split}
\end{align}
We deform the contour \(C\) into \(C'\), as shown in figure \ref{fig:3-002}, and simplify all terms to express the finite-part integral \eqref{ex4.1.1-001} as
\begin{align}
    \label{nonmel_00}
    \begin{split}
    & \bbint{0}{\infty} \frac{\sqrt{a+x}}{\sqrt{b+x}} \frac{\ln{x}}{x^{\nu + m}}\; \mathrm{d}x = - \frac{2i (-1)^{m} e^{-\pi i \nu}}{ e^{-2\pi i \nu} -1 } \int_b^a \frac{\sqrt{a-x}}{\sqrt{x-b}} \frac{\ln{x}}{x^{\nu + m}} \mathrm{d}x \\
    & \quad + \frac{2\pi (-1)^{m} e^{-\pi i \nu}}{ e^{-2\pi i \nu} -1 } \int_b^a \frac{\sqrt{a-x}}{\sqrt{x-b}} \frac{1}{x^{\nu + m}} \mathrm{d}x - \frac{4\pi (-1)^{m} e^{-3\pi i \nu}}{ ( e^{-2\pi i \nu} -1 )^2 } \int_b^a \frac{\sqrt{a-x}}{\sqrt{x-b}} \frac{1}{x^{\nu + m}} \mathrm{d}x.
    \end{split}
\end{align}

We reduce the evaluation of the finite-part integral to solving the definite integrals in \eqref{nonmel_00}. To evaluate the integrals in the second and third terms, we apply the tabulated formula from \cite[p. 331, 3.259.2]{Gradshteyn-Ryzhik}:
\begin{align}
    \label{nonmel_01}
    \int_0^u x^{\nu-1} & (u-x)^{\mu -1} (x^m + b^m)^\lambda \; \mathrm{d}x = b^{m\lambda} u^{\mu + \nu - 1} \mathrm{B}(\mu, \nu) \\
    & \times {}_{m+1}F_m \left(-\lambda, \frac{\nu}{m}, \ldots, \frac{\nu + m -1}{m} \, ; \, \frac{\mu + \nu}{m}, \ldots, \frac{\mu + \nu + m -1}{m} \, ; \, \frac{-u^m}{b^m} \right), \nonumber
\end{align}
where \(\mathrm{Re}(\mu) > 0\), \(\mathrm{Re}(\nu) > 0\), \(\mathrm{Re}(\lambda) > 0\), and \(|\arg(u/b)| < \pi/m\). Here, \({}_{m+1}F_m\) denotes the generalized hypergeometric function. It is defined by the series \cite[26.2.1]{NIST}
\begin{align}
    {}_{p}F_q(a_1, \ldots, a_p \, ; \, b_1, \ldots, b_q \, ; \, z) = \sum_{n=0}^\infty \frac{(a_1)_n \cdots (a_p)_n}{(b_1)_n \cdots (b_q)_n} \frac{z^n}{n!},
\end{align}
where \((a)_n\) denotes the Pochhammer symbol.
We perform a constant translation in \(x\) to rewrite the integral in \eqref{nonmel_00} in this form. We then obtain
\begin{align}
    \label{nonmel_02}
    \int_b^a \sqrt{\frac{a-x}{x-b}} \frac{1}{x^{\nu + m}} \mathrm{d}x = \frac{\pi}{2} \left(\frac{a-b}{b^{\nu + m}}\right) {}_2F_1 \left(\frac{1}{2}, \nu + m ; 2, 1 - \frac{a}{b}\right),
\end{align}
for \(a > b > 0\) and positive integer \(m\). To solve the integral in the first term of \eqref{nonmel_00}, we differentiate the right-hand side of \eqref{nonmel_02} with respect to \(\nu\), yielding
\begingroup
\allowdisplaybreaks
    \begin{align}
        \label{nonmel_03}
        \int_b^a \sqrt{\frac{a-x}{x-b}} \frac{\ln{x}}{x^{\nu + m}} \; \mathrm{d}x = \; \frac{\pi}{2}  \left(\frac{a-b}{b^{\nu + m}}\right) \Biggl[ \; {}_2F_1^{(0,1,0,0)} & \left(\frac{1}{2}, \nu + m ;2, 1 -\frac{a}{b}\right) \ln b \\
        - &  \, _2F_1\left(\frac{1}{2}, \nu + m ;2;1-\frac{a}{b}\right) \; \Biggr]. \nonumber
    \end{align}
\endgroup
Here, \( {}_2F_1^{(0,n,0,0)}(a, b; c; z) \) denotes the \( n \)th partial derivative of the Gauss hypergeometric function \( {}_2F_1(a, b; c; z) \) with respect to the second parameter \( b \). That is,
\begin{align}
{}_2F_1^{(0,n,0,0)}(a, b; c; z) := \frac{\partial^n}{\partial b^n} \, {}_2F_1(a, b; c; z).
\end{align}

Substituting \eqref{nonmel_02} and \eqref{nonmel_03} into \eqref{nonmel_00}, we write
\begin{align}
    \label{nonMel-ans}
    \begin{split}
    \bbint{0}{\infty} \sqrt{\frac{a+x}{b+x}} \frac{\ln{x}}{x^{\nu + m}} \, \mathrm{d}x &= (\pi^2 - \pi i \ln b) \left(\frac{a-b}{b^m}\right) \frac{e^{-\pi i \nu}}{( e^{-2\pi i \nu} -1 ) b^\nu} {}_2F_1 \left(\frac{1}{2}, \nu + m ; 2, 1 - \frac{a}{b}\right) \\
    &\quad + \pi i \left(\frac{a-b}{b^m}\right) \frac{e^{-\pi i \nu}}{( e^{-2\pi i \nu} -1 ) b^\nu} {}_2F_1^{(0,1,0,0)} \left(\frac{1}{2}, \nu + m ; 2, 1 - \frac{a}{b}\right) \\
    &\quad - 2 \pi^2 \left(\frac{a-b}{b^m}\right) \frac{e^{-3\pi i \nu}}{( e^{-2\pi i \nu} -1 )^2 b^\nu} {}_2F_1 \left(\frac{1}{2}, \nu + m ; 2, 1 - \frac{a}{b}\right).
    \end{split}
\end{align}

The last step is to find the regularized limit of equation \eqref{nonMel-ans} as \(\nu \to 0\). Let's denote the regularized limits involved in the first, second, and third terms as follows:
    \begin{align}
        \label{reglim_01}
        \lim^\times_{\nu \to 0} \mathbf{T}_1 (\nu) = \lim^\times_{\nu \to 0} \frac{e^{-\pi i \nu}}{\left( e^{-2\pi i \nu} -1 \right) b^\nu} \, {}_2F_1 \left(\frac{1}{2}, \nu + m ;2, 1 -\frac{a}{b}\right),
    \end{align}
    \begin{align}
        \label{reglim_01-1}
        \lim^\times_{\nu \to 0} \mathbf{T}_2 (\nu)=\lim^\times_{\nu \to 0} \frac{e^{-\pi i \nu}}{\left( e^{-2\pi i \nu} -1 \right) b^\nu} \, {}_2F_1^{(0,1,0,0)} \left(\frac{1}{2}, \nu + m ;2, 1 -\frac{a}{b}\right),
    \end{align}
    \begin{align}
        \label{reglim_01-2}
        \lim^\times_{\nu \to 0} \mathbf{T}_3 (\nu)=\lim^\times_{\nu \to 0} \frac{e^{-3\pi i \nu}}{( e^{-2\pi i \nu} -1 )^2 b^\nu} {}_2F_1 \left(\frac{1}{2}, \nu + m ; 2, 1 - \frac{a}{b}\right).
    \end{align}
We solve equations \eqref{reglim_01} and \eqref{reglim_01-1} using Corollary~\ref{reglim-n1}. Let
    \begin{align}
        \label{f(nu)-g(nu)}
        f(\nu) = \frac{e^{-\pi i \nu}}{ b^\nu} \, {}_2F_1 \left(\frac{1}{2}, \nu + m ;2, 1 -\frac{a}{b}\right)  \quad \text{and} \quad g(\nu) = \left( e^{-2\pi i \nu} -1 \right)
    \end{align}
for equation \eqref{reglim_01}, and
    \begin{align}
        \label{f(nu)-g(nu)-01}
        f(\nu) = \frac{e^{-\pi i \nu}}{ b^\nu} \, {}_2F_1^{(0,1,0,0)} \left(\frac{1}{2}, \nu + m ;2, 1 -\frac{a}{b}\right)  \quad \text{and} \quad g(\nu) = \left( e^{-2\pi i \nu} -1 \right)
    \end{align}
for equation \eqref{reglim_01-1}. We obtain the solutions:
    \begin{align}
        \label{reglim_02}
        \lim^\times_{\nu \to 0} \mathbf{T}_1 (\nu) = & \frac{1}{2\pi i} \, {}_2F_1 \left(\frac{1}{2}, m ;2, 1 -\frac{a}{b}\right)  - \frac{1}{2\pi i} \, {}_2F_1^{(0,1,0,0)} \left(\frac{1}{2}, m ;2, 1 -\frac{a}{b}\right),
    \end{align}
    \begin{align}
        \label{reglim_03}
        \lim^\times_{\nu \to 0} \mathbf{T}_2 (\nu) = \frac{1}{2\pi i} \, {}_2F_1^{(0,1,0,0)} \left(\frac{1}{2}, m ;2, 1 -\frac{a}{b}\right) - \frac{1}{2\pi i} \, {}_2F_1^{(0,2,0,0)} \left(\frac{1}{2}, m ;2, 1 -\frac{a}{b}\right).
    \end{align}

To determine the regularized limit of the third term of equation \eqref{nonMel-ans} using Corollary~\ref{reglim-n2}, let
    \begin{align}
        \label{f(nu)-g(nu)-02}
        f(\nu) = \frac{e^{-3\pi i \nu}}{ b^\nu} \, {}_2F_1 \left(\frac{1}{2}, \nu + m ;2, 1 -\frac{a}{b}\right),  \quad \text{and} \quad g(\nu) = \left( e^{-2\pi i \nu} -1 \right)^2.
    \end{align}
We evaluate the regularized limit of equation \eqref{f(nu)-g(nu)-02} as
\begingroup
\allowdisplaybreaks
    \begin{align}
        \label{reglim_04}
        \lim^\times_{\nu \to 0} \mathbf{T}_3 (\nu) = & -\frac{1}{8 \pi ^2} \, {}_2F_1^{(0,2,0,0)} \left(\frac{1}{2},m,2,1-\frac{a}{b}\right) \\
        & \;\;\; +\frac{1}{4 \pi ^2} \, {}_2F_1^{(0,1,0,0)}\left(\frac{1}{2},m,2,1-\frac{a}{b}\right) (\ln b + \pi i ) \nonumber\\
        & \;\;\;\;\; +\frac{1}{24 \pi ^2} \, _2F_1\left(\frac{1}{2},m;2;1-\frac{a}{b}\right) \left(-3 \ln ^2 b - 6 i \pi  \ln b +\pi ^2\right). \nonumber
    \end{align}
\endgroup

Combining the solutions from equations \eqref{reglim_02}, \eqref{reglim_03}, and \eqref{reglim_04}, we obtain the regularized limit of the original integral in equation \eqref{nonMel-ans}. The final solution is
    \begin{align}
        \label{nonMel-finalans}
        \bbint{0}{\infty} & \sqrt{\frac{a+x}{b+x}}  \frac{\ln{x}}{x^m}\; \mathrm{d}x =  \frac{(-1)^{m+1}}{12}  \left(\frac{a-b}{b^{m }}\right)  \Biggl[ 3 \; {}_2F_1^{(0,2,0,0)}\left(\frac{1}{2},m ,2,1-\frac{a}{b}\right) \\ 
        &  -6 \ln b \; {}_2F_1^{(0,1,0,0)}\left(\frac{1}{2},m ,2,1-\frac{a}{b}\right)+\left(3 \ln^2 b+\pi ^2\right) \, _2F_1\left(\frac{1}{2},m ;2;1-\frac{a}{b}\right)\Biggr]. \nonumber
    \end{align}
This is a good example to show that even if a non-Mellin-type divergent integral does not have a Mellin transform, its finite-part integral still exists.

\subsection{Numerical values of finite-part integral using contour integral representation}
As convergent integrals require numerical methods, finite-part integrals also require them. To compute numerical values of finite-part integrals, we use the contour of integration shown in Figure~\ref{fig:3-001}.

\noindent \textit{Example.} Consider finding the numerical value of the finite-part integral of the function
\begin{align}
    \label{cont-int-00}
    \bbint{0}{\infty} \frac{J_0^2(x)}{x \, \Gamma(1+x)} \mathrm{d}x.
\end{align}
Using equation~\eqref{theorem3.2-000}, we write the integral in equation~\eqref{cont-int-00} as
\begin{align}
    \label{J_02}
    \bbint{0}{\infty} \frac{J_0^2(x)}{x \, \Gamma(1+x)} \mathrm{d}x = \frac{1}{2 \pi i} \int_C \frac{J_0^2(z)}{z \, \Gamma(1+z)} \ln{z} \, \mathrm{d}z - \frac{1}{2} \int_C \frac{J_0^2(z)}{z \, \Gamma(1+z)} \mathrm{d}z.
\end{align}
We evaluate the first contour integral by deforming the contour \(C\) into a keyhole shape, as in previous examples. This gives
\begin{align}
    \label{Bess-Gam-01}
    \begin{split}
    \int_C \frac{J_0^2(z)}{z \, \Gamma(1+z)} \ln{z} \, \mathrm{d}z &= \int_\infty^\epsilon \frac{J_0^2(x)}{x \, \Gamma(1+x)} \ln{x} \, \mathrm{d}x + \int_\epsilon^\infty \frac{J_0^2(x)}{x \, \Gamma(1+x)} \ln\bigl(e^{2\pi i} x\bigr) \, \mathrm{d}x \\
    &\hspace{5cm} + \int_{C_\epsilon} \frac{J_0^2(z)}{z \, \Gamma(1+z)} \ln{z} \, \mathrm{d}z.
    \end{split}
\end{align}
Simplifying equation \eqref{Bess-Gam-01}, we obtain
\begin{align}
    \label{Bess-Gam-02}
    \int_C \frac{J_0^2(z)}{z \, \Gamma(1+z)} \ln{z} \, \mathrm{d}z = (2\pi i) \int_\epsilon^\infty \frac{J_0^2(x)}{x \, \Gamma(1+x)} \mathrm{d}x + \int_{C_\epsilon} \frac{J_0^2(z)}{z \, \Gamma(1+z)} \ln{z} \, \mathrm{d}z.
\end{align}
Setting \(\epsilon = 1\) and expressing \(z\) in exponential form, \(z = \epsilon e^{i\theta}\), we have
\begin{align}
    \label{Bess-Gam-03}
    \int_C \frac{J_0^2(z)}{z \, \Gamma(1+z)} \ln{z} \, \mathrm{d}z = (2\pi i) \int_1^\infty \frac{J_0^2(x)}{x \, \Gamma(1+x)} \mathrm{d}x - \int_0^{2\pi} \frac{\theta \, J_0^2(e^{i\theta})}{\Gamma(1+e^{i\theta})} \mathrm{d}\theta.
\end{align}
The second contour integral in equation~\eqref{J_02} evaluates to
\begin{align}
    \label{Bess-Gam-04}
    \int_C \frac{J_0^2(z)}{z \, \Gamma(1+z)} \mathrm{d}z = i \int_0^{2\pi} \frac{J_0^2(e^{i\theta})}{\Gamma(1+e^{i\theta})} \mathrm{d}\theta.
\end{align}

Substituting equations~\eqref{Bess-Gam-03} and \eqref{Bess-Gam-04} into equation~\eqref{J_02} yields
\begin{align}
    \label{J0_04}
    \begin{split}
    \bbint{0}{\infty} \frac{J_0^2(x)}{x \, \Gamma(1+x)} \mathrm{d}x &= \int_1^\infty \frac{J_0^2(x)}{x \, \Gamma(1+x)} \mathrm{d}x - \frac{1}{2\pi i} \int_0^{2\pi} \frac{\theta \, J_0^2(e^{i\theta})}{\Gamma(1+e^{i\theta})} \mathrm{d}\theta \\
    &\hspace{5cm} - \frac{i}{2} \int_0^{2\pi} \frac{J_0^2(e^{i\theta})}{\Gamma(1+e^{i\theta})} \mathrm{d}\theta.
    \end{split}
\end{align}
We can evaluate all integrals in equation~\eqref{J0_04} by any numerical method. Using Mathematica, we obtain the final numerical value:
\begin{align}
    \bbint{0}{\infty} \frac{J_0^2(x)}{x \, \Gamma(1+x)} \mathrm{d}x = 0.2129210647.
\end{align}
Although finding a closed-form analytical solution to this problem is difficult, solving the numerical values of the finite-part integral using the contour integral representation works even for more complex functions.


\section{Finite-Part Integration of the Stieltjes Transform with Logarithmic Singularity}
\label{Stieljes-Transform-Section}

We consider the generalized Stieltjes transform with logarithmic singularities given by
\begin{equation}
    \label{sec4.2-003}
    \int_0^a \frac{k(t) \ln^n{t}}{t^\nu (\omega^2 + t^2)} \; \mathrm{d}t,
\end{equation}
for \(0 \leq \mathrm{Re} (\nu) < 1\), \(-\pi < \mathrm{Arg} (\omega) \leq \pi \) with \(|\mathrm{Arg} (\omega)| \neq \pi/2 \), and \(0 < a \leq \infty\). The next two sections of this chapter examine distinct cases of the integral in equation~\eqref{sec4.2-003}.

\subsection{Stieltjes transform in the case \texorpdfstring{\(\nu \neq 0\)}{nu neq 0}}
    \begin{theorem}
        \label{theorem4.1}
        Let \(k(t)\) satisfy the following properties: (1) \(k(z)\) is its complex extension and it is analytic in the interval \([0,a)\); (2) when evaluated at \(t=0\), \(k(0) \neq 0\), then
        \begingroup
        \allowdisplaybreaks
            \begin{align}
                \label{th4.1-001}
                \begin{split}
                & \int_0^a \frac{k(t) \ln^n{t}}{t^\nu (\omega^2 + t^2)} \mathrm{d}t =  \sum_{k=0}^\infty (-1)^k \omega^{2k}  \bbint{0}{a} \frac{k(t) \ln^n{t}}{t^{\nu +2k+2}} \mathrm{d}t  \\
                & \hspace{1cm}+ \biggl[ \; \frac{k(i \omega) + k(-i \omega)}{2} \; \biggr] \Delta_{n_1} (\nu, \omega)- i \, \biggl[ \; \frac{k(i \omega) - k(-i \omega)}{2} \; \biggr] \Delta_{n_2} (\nu, \omega),
                \end{split}
            \end{align}
        where
            \begin{align}
                \label{extra-term}
                \Delta_{n_1} (\nu, \omega) = \frac{\pi}{2\omega^{\nu +1}} \sum_{j=0}^n (-1)^j \binom{n}{j} (\mathrm{Log} \; \omega)^{n-j} \frac{\mathrm{d}^j}{\mathrm{d}\nu^j}\sec \left(\frac{\pi \nu}{2} \right),
            \end{align}
            \begin{align}
                 & \frac{\mathrm{d}^j}{\mathrm{d}\nu^j}\sec \left(\frac{\pi \nu}{2} \right) = (-1)^{[(j+1)/2]}\left(\frac{\pi}{2}\right)^j \sec \left( \frac{\pi \nu}{2} \right) \sum_{l=0}^j \, \frac{1}{2^l} \, \sum_{m=0}^{[(l-\gamma)/2]} (-1)^m \binom{l}{2m+\gamma} \\
                 & \hspace{2cm} \times \tan^{2m+\gamma} \left(\frac{\pi \nu}{2}\right) \sum_{p=0}^l (-1)^p \binom{l}{p}(2p+1)^j, \;\;\; \gamma = \frac{1-(-1)^j}{2}, \nonumber
            \end{align}
        and
            \begin{align}
                \label{extra-term-sec}
                \Delta_{n_2} (\nu, \omega) = \frac{\pi}{2\omega^{\nu +1}} \sum_{j=0}^n (-1)^j \binom{n}{j} (\mathrm{Log} \; \omega)^{n-j} \frac{\mathrm{d}^j}{\mathrm{d}\nu^j}\csc \left(\frac{\pi \nu}{2} \right),
            \end{align}
            \begin{align}
                \begin{split}
                 & \frac{d^j}{d\nu^j} \csc \left( \frac{\pi \nu}{2} \right) = (-1)^{[j/2]} \left( \frac{\pi}{2} \right)^j \csc \left(\frac{\pi \nu}{2} \right) \sum_{l=1}^j \, \frac{1}{2^l} \, \sum_{m=0}^l (-1)^m \binom{l}{m} (2m +1)^j \\
                 & \hspace{2cm} \times \sum_{p=0}^{[(l-\gamma)/2]} (-1)^p \binom{l}{2p+\gamma} \cot^{2p+\gamma}\left(\frac{\pi \nu}{2}\right), \;\;\; \gamma = \frac{1-(-1)^j}{2},
                 \end{split}
            \end{align}
        \endgroup
        for \(|\omega| < \mathrm{min}(a, \rho_o)\), \(-\pi < \mathrm{Arg} (\omega) \leq \pi \) with \(|\mathrm{Arg} (\omega)| \neq \pi/2 \), \(0<\mathrm{Re} (\nu) <1\), and for all positive integer n, where \(\omega^{-\nu}\) takes its principal value and \(\rho_o\) is defined as the distance of the nearest singularity of \(k(z)\) from the origin. Equation \eqref{th4.1-001} still holds for \(a = \infty\), if the Stieltjes integral exists as \(a \to \infty\) for all \(|\omega| < \infty\) when \(k(z)\) is entire or for all \(|\omega| < \rho_o\) when \(k(z)\) has at least one singularity. Lastly, it holds that
        \begin{align}
            \label{th4.1-001-asym}
            \int_0^a \frac{k(t) \ln^n{t}}{t^\nu (\omega^2 + t^2)} \mathrm{d}t \; \sim \;  k(0) \; \Delta_{n_1} (\nu, \omega), \;\;\;\; \omega \to 0.
        \end{align}
    \end{theorem}

\begin{proof}
For now, let's assume that \( \omega > 0 \), and subsequently extend the result to more general values of \( \omega \) through analytic continuation.  We first consider the case \(a < \infty\). The Stieltjes transform can be evaluated by considering the contour integral
\begin{equation}
    \label{cont-int-logn}
    \int_C \frac{k(z) \ln^n z}{z^\nu (\omega^2 + z^2)} \, \mathrm{d}z,
\end{equation}
where the contour \(C\) is illustrated in Figure~\ref{fig:3-003}. The contour \( C \) must be chosen to satisfy the following criteria: (i) The singularities of the kernel at \( z = \pm i \omega \) should lie to the left of the contour when it is traversed in the same direction; (ii) All singularities of \( k(z) \) must remain to the right of the contour when traversed in the same direction. The first condition ensures the validity of term-by-term integration, while the second guarantees that the resulting contour integrals can be interpreted as finite-part integrals \cite{galapon2023}.

We deform the contour \(C\) into \(C'\) in equation~\eqref{cont-int-logn}, accounting for the poles introduced by the term \((\omega^2 + z^2)^{-1}\). This deformation yields
\begingroup
\allowdisplaybreaks
\begin{align}
    \label{cont-int-logn-01}
    &\int_C \frac{k(z) \ln^n z}{z^\nu (\omega^2 + z^2)} \, \mathrm{d}z 
    = \left( e^{-2\pi i \nu} - 1 \right) \int_0^a \frac{k(t) \ln^n t}{t^\nu (\omega^2 + t^2)} \, \mathrm{d}t \\ 
    &\hspace{1cm} + \; e^{-2\pi i \nu} \sum_{j=0}^{n-1} \binom{n}{j} \left( 2\pi i \right)^{n-j} \int_0^a \frac{k(t) \ln^j t}{t^\nu (\omega^2 + t^2)} \, \mathrm{d}t \nonumber \\
    &\hspace{1cm} + \; k(i\omega) \left[ \frac{\pi  e^{-\pi i \nu /2} }{\omega^{\nu + 1}}  \left( \ln \omega + \frac{\pi i}{2} \right)^n \right] 
    - k(-i\omega) \left[ \frac{\pi e^{-3\pi i \nu /2} }{\omega^{\nu + 1}} \left( \ln \omega + \frac{3 \pi i}{2} \right)^n \right]. \nonumber
\end{align}
\endgroup
    \begin{figure}[t]
        \centering
        \includegraphics[width=10cm]{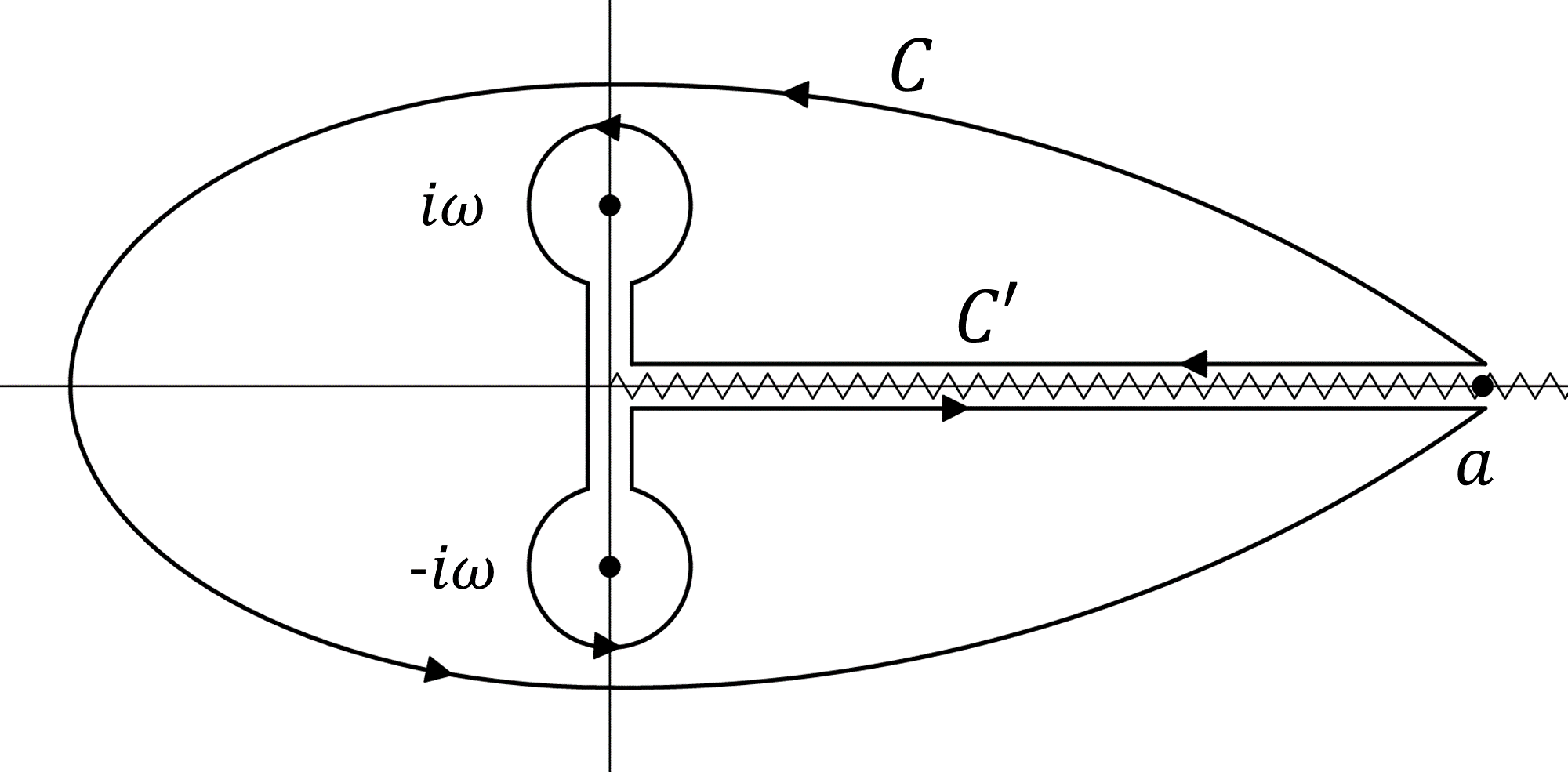}
        \caption{Contour C deformed into contour C' considering the poles coming from \((\omega^2 +z^2)^{-1}\). The branch cut for \(z^\lambda\) is along the positive real axis. }
        \label{fig:3-003}
    \end{figure}

Using equation~\eqref{cont-int-logn-01} and applying the series expansion
\begin{align}
    \label{ser-expand}
    \frac{1}{\omega^2 + z^2} = \sum_{k=0}^\infty (-1)^k \frac{\omega^{2k}}{z^{2k+2}}, \qquad \left| \frac{\omega^2}{z^2}\right| < 1,
\end{align}
we can express the associated Stieltjes transform in the form
\begingroup
\allowdisplaybreaks
\begin{align}
    \label{stieltjes-logn-01}
    \begin{split}
    \int_0^a \frac{k(t) \ln^n t}{t^\nu (\omega^2 + t^2)} \mathrm{d}t &= \sum_{k=0}^\infty (-1)^k \omega^{2k} \frac{1}{e^{-2\pi i \nu} - 1} \int_C \frac{k(z) \ln^n z}{z^{\nu + 2k + 2}} \mathrm{d}z \\
    &\hspace{0.5cm} - \sum_{j=0}^{n-1} \binom{n}{j} (2\pi i)^{n-j} \frac{e^{-2\pi i \nu}}{e^{-2\pi i \nu} - 1} \int_0^a \frac{k(t) \ln^j t}{t^\nu (\omega^2 + t^2)} \mathrm{d}t \\
    &\hspace{0.8cm} - k(i\omega) \left[ \frac{\pi e^{-\pi i \nu /2}}{\omega^{\nu + 1}(e^{-2\pi i \nu} - 1)} \left( \ln \omega + \frac{\pi i}{2} \right)^n \right] \\
    &\hspace{1cm} + k(-i\omega) \left[ \frac{\pi e^{-3\pi i \nu /2}}{\omega^{\nu + 1}(e^{-2\pi i \nu} - 1)} \left( \ln \omega + \frac{3\pi i}{2} \right)^n \right].
    \end{split}
\end{align}
\endgroup
Under the condition \( \omega < |z| \), the series expansion in equation~\eqref{ser-expand} converges uniformly along the contour \( C \).  
This uniform convergence justifies the interchange of summation and integration in equation~\eqref{stieltjes-logn-01}, allowing for term-by-term integration.

To solve equation~\eqref{stieltjes-logn-01} recursively, we begin by isolating the final term in the second summation, corresponding to \( j = n - 1 \). We then express the isolated integral using equation~\eqref{stieltjes-logn-01} itself, thereby introducing additional terms. This step produces two new summations, both with upper limits reduced to \( j = n - 2 \). By repeating this process iteratively, we eventually reduce all upper limits of summation to zero. Each resulting term contains coefficients that follow the same combinatorial structure observed in the analysis of \( G^n_l(z) \) in Theorem~\ref{contour-noninteger}. The Stieltjes transform then admits the representation
\begingroup
\allowdisplaybreaks
\begin{align}
\label{stielt-trans-t1}
    \int_0^a\frac{k(t) \ln^n t}{t^\nu (\omega^2 + t^2)} \, \mathrm{d}t &= \sum_{k=0}^\infty (-1)^k \omega^{2k} 
    \left[ \sum_{j=0}^n \binom{n}{j} (2 \pi i)^{n-j} \beta_j(\nu) \int_C \frac{k(z) \ln^j z}{z^{\nu + 2k + 2}} \, \mathrm{d}z \right] \\
    &\hspace{0.7cm} - k(i\omega) \left[ \frac{\pi e^{- \pi i \nu /2}}{\omega^{\nu +1}} \sum_{j=0}^n \binom{n}{j} (2 \pi i)^{n-j} \beta_j(\nu) \left( \ln \omega + \frac{\pi i}{2} \right)^j \right] \nonumber \\
    &\hspace{0.7cm} + k(-i\omega) \left[ \frac{\pi e^{-3 \pi i \nu /2}}{\omega^{\nu +1}} \sum_{j=0}^n \binom{n}{j} (2 \pi i)^{n-j} \beta_j(\nu) \left( \ln \omega + \frac{3\pi i}{2} \right)^j \right], \nonumber
\end{align}
\endgroup
where
\begin{align}
    \beta_j(\nu) = \sum_{l=0}^{n-j} (-1)^l \, l! \, \mathbf{S}^{n-j}_l \frac{(e^{-2\pi i \nu})^l }{(e^{-2\pi i \nu } - 1)^{l +1}}.
\end{align}

Equation~\eqref{stielt-trans-t1} consists of three terms. Using the definition of the finite-part integral in equation~\eqref{theorem3.1}, the inner sum containing all contour integrals in the first term corresponds to the finite part of a divergent integral with a logarithmic singularity, valid for non-integer values of \(\lambda\). The remaining two terms represent residues, which take the form
\begingroup
\allowdisplaybreaks
\begin{align}
    & - k(i\omega) \; \frac{(-1)^n \pi }{\omega}\frac{\mathrm{d}^n}{\mathrm{d}\nu ^n} \left[ \frac{(\omega e^{\pi i/2})^{-\nu}}{e^{-2 \pi i \nu}-1} \right] 
    = - k(i\omega) \; \frac{(-1)^n \pi i }{2 \omega} \frac{\mathrm{d}^n}{\mathrm{d}\nu ^n} \left[ \omega^{-\nu} \csc(\pi \nu) \, e^{\pi i \nu/2} \right] \\
    & \hspace{2cm} = k(i\omega) \; \frac{(-1)^n \pi  }{4 \omega} \frac{\mathrm{d}^n}{\mathrm{d}\nu ^n} \left[ \omega^{-\nu} \sec\left( \frac{\pi \nu}{2} \right) - i \, \omega^{-\nu} \csc\left( \frac{\pi \nu}{2} \right) \right], \nonumber \\
    & k(-i \omega) \; \frac{(-1)^n \pi }{\omega} \frac{\mathrm{d}^n}{\mathrm{d}\nu ^n} \left[ \frac{(\omega e^{3\pi i/2})^{-\nu}}{e^{-2 \pi i \nu}-1} \right] 
    = k(-i \omega) \; \frac{(-1)^n \pi i }{2 \omega} \frac{\mathrm{d}^n}{\mathrm{d}\nu ^n} \left[ \omega^{-\nu} \csc(\pi \nu) \, e^{-\pi i \nu/2} \right] \\
    & \hspace{2cm} = k(-i \omega) \; \frac{(-1)^n \pi  }{4 \omega} \frac{\mathrm{d}^n}{\mathrm{d}\nu ^n} \left[ \omega^{-\nu} \sec\left( \frac{\pi \nu}{2} \right) + i \, \omega^{-\nu} \csc\left( \frac{\pi \nu}{2} \right) \right], \nonumber
\end{align}
\endgroup
as derived from Theorem~\ref{nth-der}. Grouping similar terms prior to differentiation and applying the Leibniz rule leads to equations~\eqref{extra-term} and~\eqref{extra-term-sec}, where the \(j^\text{th}\) derivatives of \(\sec(\pi \nu/2)\) and \(\csc(\pi \nu/2)\) are taken from~\cite[p. 8, 1.1.5.5, 1.1.5.9]{Brychkov}.

The natural logarithm \(\ln \omega\) is replaced by the complex principal value logarithm \(\mathrm{Log} \; \omega\) in order to analytically continue the function from the positive real axis into the complex plane. When \(\omega\) is complex, the left-hand side of equation~\eqref{stielt-trans-t1} remains analytic provided that \(-\pi < \mathrm{Arg} (\omega) \leq \pi \) with \(|\mathrm{Arg} (\omega)| \neq \pi/2 \).

Establishing the convergence of the series in equation~\eqref{th4.1-001} is essential to prove this theorem. The case when \( a < \infty \) has two possible scenarios for the finite value of \( a \), namely \( a < \rho_o \) and \( a > \rho_o \). We begin with the case \( a > \rho_o \). Recall that the parameter \( \epsilon \) can be taken arbitrarily small, provided that \( \epsilon < \rho_o \) holds. As shown in~\cite{galapon2023}, the following inequality holds:
\begin{align}
    \label{th4.1-002}
    \left| \bbint{0}{a} \frac{k(t) \ln^n{t}}{t^{r+\nu+1}} \, \mathrm{d}t \right| 
    \leq \frac{M_\nu(a,\epsilon)}{\epsilon^{r+\nu}},
\end{align}
where \( r \) is a positive integer and \( M_\nu(a,\epsilon) \) is defined by
\begin{align}
    M_\nu(a,\epsilon) = 
    \int_\epsilon^a \frac{k(t) \ln^n{t}}{t^{\nu + k}} \, \mathrm{d}t 
    + n! \sum_{m=0}^n \frac{|\ln \epsilon|^m \, M(\nu)^{n - m + 1}}{m!},
\end{align}
with the auxiliary function \( M(\nu) \) given by
\begin{align}
    M(\nu) = 
    \begin{cases}
        \displaystyle \frac{2\sinh{(\pi \, \mathrm{Im}(\nu))}}{
        \mathrm{Im}(\nu) \sqrt{\sin^2(\pi \, \mathrm{Re}(\nu)) 
        + \sinh^2{(\pi \, \mathrm{Im}(\nu))}}}, & \text{if } \mathrm{Im}(\nu) \neq 0, \\[1em]
        \displaystyle \frac{2\pi}{|\sin{(\pi \nu)}|}, & \text{if } \mathrm{Im}(\nu) = 0.
    \end{cases}
\end{align}

Given the bound in equation~\eqref{th4.1-002}, the absolute convergence of the infinite series in equation~\eqref{th4.1-001} follows from the inequality
\begin{align}
    \label{th4.1-003}
    \left| \sum_{k=0}^\infty (-1)^k \omega^{2k}  
    \bbint{0}{a} \frac{k(t) \ln^n{t}}{t^{\nu + 2k + 2}} \, \mathrm{d}t \right| 
    \leq \sum_{k=0}^\infty |\omega|^{2k} 
    \left| \bbint{0}{a} \frac{k(t) \ln^n{t}}{t^{\nu + 2k + 2}} \, \mathrm{d}t \right|.
\end{align}
Substituting the bound from equation~\eqref{th4.1-002} into the right-hand side of equation~\eqref{th4.1-003}, we obtain
\begin{align}
    \label{th4.1-004}
    \left| \sum_{k=0}^\infty (-1)^k \omega^{2k}  
    \bbint{0}{a} \frac{k(t) \ln^n{t}}{t^{\nu + 2k + 2}} \, \mathrm{d}t \right| 
    \leq \frac{M_\nu(a,\epsilon)}{\epsilon^{\nu + 1}} 
    \sum_{k=0}^\infty \left| \frac{\omega}{\epsilon} \right|^{2k}.
\end{align}
The convergence of the infinite series on the left-hand side is therefore guaranteed by the convergence of the geometric series on the right-hand side. This series converges provided that \( |\omega| < \epsilon < \rho_o \), or equivalently \( |\omega| < \rho_o < a \), given that \( \rho_o < a \).

The function \(k(t)\) remains analytic when the upper limit of integration satisfies \(a < \rho_0\), as no singularity lies within the contour. This condition permits the use of the power series expansion of \(k(t)\) about \(t = 0\), allowing for term-by-term integration. To establish the absolute convergence of the infinite series in equation~\eqref{th4.1-001}, we examine the bound
\begin{align}
    \label{th4.1-005}
    \left| \bbint{0}{a} \frac{k(t) \ln^n{t}}{t^{\nu + 2k + 2}} \, \mathrm{d}t \right|
    \leq \left| \int_\epsilon^a \frac{k(t) \ln^n{t}}{t^{\nu + 2k + 2}} \, \mathrm{d}t \right| 
    + n! \sum_{m=0}^n \frac{|\ln \epsilon|^m}{m!} 
    \sum_{l=0}^\infty |c_l| \frac{\epsilon^{l - \nu - 2k - 1}}{|l - \nu - 2k - 1|^{n - m + 1}}.
\end{align}
The first term on the right-hand side of equation~\eqref{th4.1-005} provides the dominant contribution. By expanding \(k(t)\) as \(k(t) = \sum_{l=0}^\infty c_l t^l\), we obtain the bound
\begin{align}
    \label{th4.1-006}
    \left| \int_\epsilon^a \frac{k(t) \ln^n{t}}{t^{\nu + 2k + 2}} \, \mathrm{d}t \right|
    \leq \sum_{l = 0}^{2k + 2} |c_l| \int_\epsilon^a \left| \frac{\ln^n{t}}{t^{\nu + 2k + 2 - l}} \right| \, \mathrm{d}t
    + \sum_{l = 2k + 3}^{\infty} |c_l| \int_\epsilon^a \left| t^{l - \nu - 2k - 2} \ln^n{t} \right| \, \mathrm{d}t.
\end{align}
As \(l\) increases, the exponent of \(t\) in the second sum becomes increasingly positive, causing this term to dominate. We estimate it using the inequality
\begin{align}
    \label{th4.1-007}
    \int_\epsilon^a \left| t^{l - \nu - 2k - 2} \ln^n{t} \right| \, \mathrm{d}t 
    \leq \frac{1}{a^{2k + \nu}} \int_\epsilon^a \left| t^{l - 2} \ln^n{t} \right| \, \mathrm{d}t.
\end{align}
By applying the same bounding strategy as in the earlier case and using equation~\eqref{th4.1-007}, we conclude that the bound on the infinite series in equation~\eqref{th4.1-001} is proportional to the geometric series \( \sum_{k=0}^\infty \left| \omega / a \right|^{2k} \). Therefore, the series converges absolutely under the condition \( |\omega| < \min(\rho_0, a) \).

Now consider the case when \(a = \infty\). Assuming that the generalized Stieltjes transform in equation~\eqref{sec4.2-003} exists, the absolute convergence of the infinite series in equation~\eqref{th4.1-001} must yield a finite value. Specifically, it must satisfy \(M_\nu(a, \epsilon) < M_\nu(\infty, \epsilon) < \infty\), and the sum involving \(M_\nu(\infty, \epsilon)\) must converge. Hence, equation~\eqref{th4.1-001} remains valid in the limit as \(a \to \infty\) and can therefore be employed in addressing the problem above.

The asymptotic relation given in equation~\eqref{th4.1-001-asym} becomes evident upon rewriting the result in equation~\eqref{th4.1-001} as
\begin{align}
    \label{asymp}
    \int_0^a \frac{k(t) \ln^n t}{t^\nu (\omega^2 + t^2)} \, \mathrm{d}t 
    = \left( \bbint{0}{a} \frac{k(t) \ln^n t}{t^{\nu + 2}} \, \mathrm{d}t + O(\omega) \right) 
    + \left( k(0) + O(i\omega) \right) \Delta_{n_1}(\nu, \omega),
\end{align}
as \(\omega \to 0\). From equation~\eqref{extra-term}, it follows that \(\Delta_{n_1}(\nu, \omega) = O(\omega^{-1-\nu} \, \mathrm{Log}^n \omega)\) as \(\omega \to 0\). This implies that the second term dominates both the first and third terms in equation~\eqref{asymp}, thereby establishing the asymptotic form stated in equation~\eqref{th4.1-001-asym}. The structure of the proof presented here follows the same approach as that given in~\cite{galapon2023}.
\end{proof}

\subsection{Stieltjes transform in the case \texorpdfstring{\(\nu = 0\)}{nu = 0}}
    \begin{theorem}
        \label{theorem4.2}
        Let \(k(t)\) satisfy the following properties: (1) \(k(z)\) is its complex extension and it is analytic in the interval \([0,a)\); (2) when evaluated at \(t=0\), \(k(0) \neq 0\), then
            \begin{align}
                \label{th4.2-001}
                & \int_0^a \frac{k(t) \ln^n{t}}{ (\omega^2 + t^2)} \mathrm{d}t =  \sum_{k=0}^\infty (-1)^k \omega^{2k}  \bbint{0}{a} \frac{k(t) \ln^n{t}}{t^{2k +2}} \mathrm{d}t \\
                & \hspace{1cm} + \biggl[ \; \frac{k(i \omega) + k(-i \omega)}{2} \; \biggr] \Delta_{n_1} (\omega)- i \, \biggl[ \; \frac{k(i \omega) - k(-i \omega)}{2} \; \biggr] \Delta_{n_2} (\omega), \nonumber
            \end{align}
where
            \begin{align}
                \label{reglim-sec}
                \Delta_{n_1} (\omega) = \frac{n!}{\omega} \sum_{j=0}^{\left\lfloor n/2 \right\rfloor} \frac{(-1)^j}{(n-2j)! \, (2j)!} \left( \frac{\pi}{2}\right)^{2j+1} E_{2j} \, (\mathrm{Log} \; \omega)^{n-2j} ,
            \end{align}
            \begin{align}
                \label{reglim-csc}
                \Delta_{n_2} (\omega) = - \frac{\mathrm{Log}^{n+1} \omega}{\omega \, (n+1)} + \frac{n!}{\omega} \sum_{j=1}^{\left\lceil n/2 \right\rceil} \frac{(-1)^j (2^{2j}-2)}{(n-2j+1)! \, (2j)!} \left( \frac{\pi}{2}\right)^{2j} B_{2j} \, (\mathrm{Log} \; \omega)^{n-2j+1},
            \end{align}
            for \(|\omega| < \mathrm{min}(a, \rho_o)\), \(-\pi < \mathrm{Arg} (\omega) \leq \pi \) with \(|\mathrm{Arg} (\omega)| \neq \pi/2 \), \(0<\mathrm{Re} (\nu) <1\), and for all positive integer n, where \(\rho_o\) is defined as the distance of the nearest singularity of \(k(z)\) from the origin. Moreover, \(E_{2j}\)'s and \(B_{2j}\)'s are the Euler numbers and Bernoulli numbers, respectively. Also, Equation \eqref{th4.2-001} still holds for \(a = \infty\), if the Stieltjes integral exists as \(a \to \infty\) for \(|\omega| < \infty\) when \(k(z)\) has no singularity or \(|\omega| < \rho_o\) when \(k(z)\) has singularities. Lastly, it holds that
        \begin{align}
            \label{th4.1-002-asym}
            \int_0^a \frac{k(t) \ln^n{t}}{(\omega^2 + t^2)} \mathrm{d}t \sim   k(0) \, \Delta_{n_1} (\omega), \;\;\;\; \omega \to 0.
        \end{align}
    \end{theorem}

\begin{proof}
We begin with the case \(a < \infty\) by considering the contour integral
\begin{equation}
    \label{cont-int-logn-nu-zero}
    \int_C \frac{k(z) \ln^{n+1} z}{\omega^2 + z^2} \, \mathrm{d}z,
\end{equation}
where the contour \(C\) is the same contour shown in Figure~\ref{fig:3-003}. We deform the contour \(C\) into contour \(C'\) in order to expand equation~\eqref{cont-int-logn-nu-zero} while taking into account the simple poles arising from \((\omega^2 + z^2)^{-1}\). This yields
\begingroup
\allowdisplaybreaks
\begin{align}
    \label{cont-int-logn-nuzero-01}
    & \int_C \frac{k(z) \ln^{n+1} z}{\omega^2 + z^2} \, \mathrm{d}z 
    = 2 \pi i \, (n+1) \int_0^a \frac{k(t) \ln^n t}{\omega^2 + t^2} \, \mathrm{d}t \\
    &\hspace{1cm} + \sum_{j=0}^{n-1} \binom{n+1}{j} (2\pi i)^{n+1-j} \int_0^a \frac{k(t) \ln^j t}{\omega^2 + t^2} \, \mathrm{d}t \nonumber \\
    &\hspace{1cm} + k(i\omega) \left[ \frac{\pi}{\omega} \left( \ln \omega + \frac{\pi i}{2} \right)^{n+1} \right] - k(-i\omega) \left[ \frac{\pi}{\omega} \left( \ln \omega + \frac{3\pi i}{2} \right)^{n+1} \right]. \nonumber
\end{align}
\endgroup
Substituting the series expansion provided in equation~\eqref{ser-expand} into the result above, we obtain an expression for the associated Stieltjes transform in the case \(\nu = 0\), given by
\begingroup
\allowdisplaybreaks
\begin{align}
    \label{stieltjes-logn-nuzer-01}
    \int_0^a \frac{k(t) \ln^n t}{\omega^2 + t^2} \, \mathrm{d}t 
    &= \frac{1}{2 \pi i \, (n+1)} \sum_{k=0}^\infty (-1)^k \omega^{2k} \int_C \frac{k(z) \ln^{n+1} z}{z^{2k+2}} \, \mathrm{d}z \\
    &\quad - \frac{1}{n+1} \sum_{j=0}^{n-1} \binom{n+1}{j} (2\pi i)^{n-j} \int_0^a \frac{k(t) \ln^j t}{\omega^2 + t^2} \, \mathrm{d}t \nonumber \\
    &\quad + k(i\omega) \left[ \frac{i}{2\omega (n+1)} \left( \ln \omega + \frac{\pi i}{2} \right)^{n+1} \right] \nonumber \\
    &\quad - k(-i\omega) \left[ \frac{i}{2\omega (n+1)} \left( \ln \omega + \frac{3\pi i}{2} \right)^{n+1} \right]. \nonumber
\end{align}
\endgroup
Again, the uniform convergence of equation~\eqref{ser-expand} along the contour \( C \) permits the interchange of summation and integration in equation~\eqref{stieltjes-logn-nuzer-01}, allowing term-by-term integration.

To evaluate equation \eqref{stieltjes-logn-nuzer-01} recursively, following the approach used for the noninteger \(\lambda\) case, we begin by isolating the last term in the second summation corresponding to \(j = n - 1\). We then express the isolated integral by equation \eqref{stieltjes-logn-nuzer-01} itself, which introduces two new summations with upper limits equal to \(j = n - 2\) and some additional terms. Repeating this process iteratively, we eventually reduce the upper limits of all summations to zero. All coefficients generated during this procedure exhibit the same combinatorial structure as described in Theorem~\ref{theorem_3.2}. Consequently, the Stieltjes transform admits the representation
\begingroup
\allowdisplaybreaks
\begin{align}
    \label{stieltjes-logn-nuzer-02}
    &\int_0^a \frac{k(t) \ln^n t}{\omega^2 + t^2} \, \mathrm{d}t 
    = \frac{1}{2 \pi i \, (n+1)} \sum_{k=0}^\infty (-1)^k \omega^{2k} \int_C \frac{k(z) \ln^{n+1} z}{z^{2k+2}} \, \mathrm{d}z \\
    &\hspace{2cm} + \sum_{j=0}^n \binom{n}{j} (2\pi i)^{n-j} \frac{B_{n-j+1}}{n-j+1} \sum_{k=0}^\infty (-1)^k \omega^{2k} \int_C \frac{k(z) \ln^j z}{z^{2k+2}} \, \mathrm{d}z \nonumber \\
    & - k(i\omega) \frac{\pi}{\omega} \left[ \frac{1}{2\pi i (n+1)} \left( \ln \omega + \frac{\pi i}{2} \right)^{n+1} + \sum_{j=0}^n \binom{n}{j} (2\pi i)^{n-j} \frac{B_{n-j+1}}{n-j+1} \left( \ln \omega + \frac{\pi i}{2} \right)^j \right] \nonumber \\
    & + k(-i\omega) \frac{\pi}{\omega} \left[ \frac{1}{2\pi i (n+1)} \left( \ln \omega + \frac{3 \pi i}{2} \right)^{n+1} + \sum_{j=0}^n \binom{n}{j} (2\pi i)^{n-j} \frac{B_{n-j+1}}{n-j+1} \left( \ln \omega + \frac{3 \pi i}{2} \right)^j \right], \nonumber
\end{align}
\endgroup
where \(B_{n-j+1}\) denotes the Bernoulli numbers.

Equation \eqref{stieltjes-logn-nuzer-02} consists of three terms. Applying the definition of the finite-part integral given in equation \eqref{theorem3.2}, the inner sum of all contour integrals corresponds to the finite-part of a divergent integral exhibiting a logarithmic singularity for integer \(\lambda\). The two residue terms can be expressed as
\begingroup
\allowdisplaybreaks
\begin{align}
    & - k(i\omega) \; \frac{(-1)^n \pi }{\omega} \lim^\times_{\nu \to 0} \frac{\mathrm{d}^n}{\mathrm{d}\nu^n} \left[ \frac{(\omega e^{\pi i/2})^{-\nu}}{e^{-2 \pi i \nu} - 1} \right] \\
    & \hspace{1cm} =  k(i\omega) \; \frac{(-1)^n \pi }{4 \omega} \lim^\times_{\nu \to 0} \frac{\mathrm{d}^n}{\mathrm{d}\nu^n} \left[ \omega^{-\nu} \sec\left(\frac{\pi \nu}{2}\right) - i \, \omega^{-\nu} \csc\left(\frac{\pi \nu}{2}\right) \right], \nonumber \\
    & k(-i\omega) \; \frac{(-1)^n \pi}{\omega} \lim^\times_{\nu \to 0} \frac{\mathrm{d}^n}{\mathrm{d}\nu^n} \left[ \frac{(\omega e^{3 \pi i/2})^{-\nu}}{e^{-2 \pi i \nu} - 1} \right] \\
    & \hspace{2cm} = k(-i\omega) \; \frac{(-1)^n \pi }{4 \omega} \lim^\times_{\nu \to 0} \frac{\mathrm{d}^n}{\mathrm{d}\nu^n} \left[ \omega^{-\nu} \sec\left(\frac{\pi \nu}{2}\right) + i \, \omega^{-\nu} \csc\left(\frac{\pi \nu}{2}\right) \right], \nonumber
\end{align}
\endgroup
as established in Theorem~\ref{nth-der-reglim}. This shows that the residue terms can be represented by the regularized limits of the grouped residue terms \(\Delta_{n_1}(\nu, \omega)\) and \(\Delta_{n_2}(\nu, \omega)\) from Theorem~\ref{theorem4.1} as \(\nu \to 0\). More precisely,
\begin{align}
    \label{d-sec}
    \Delta_{n_1}(\omega) &= \lim^\times_{\nu \to 0} \Delta_{n_1}(\nu, \omega) \\
    &= \frac{\pi}{2 \omega} \sum_{j=0}^n (-1)^j \binom{n}{j} (\mathrm{Log} \, \omega)^{n-j} \lim^\times_{\nu \to 0} \left[ \omega^{-\nu} \frac{\mathrm{d}^j}{\mathrm{d}\nu^j} \sec \left(\frac{\pi \nu}{2}\right) \right], \nonumber \\
    \label{d-csc}
    \Delta_{n_2}(\omega) &= \lim^\times_{\nu \to 0} \Delta_{n_2}(\nu, \omega) \\
    &= \frac{\pi}{2 \omega} \sum_{j=0}^n (-1)^j \binom{n}{j} (\mathrm{Log} \, \omega)^{n-j} \lim^\times_{\nu \to 0} \left[ \omega^{-\nu} \frac{\mathrm{d}^j}{\mathrm{d}\nu^j} \csc \left(\frac{\pi \nu}{2}\right) \right] \nonumber
\end{align}
where the natural logarithm \(\ln \omega\) is replaced by the complex principal value logarithm \(\mathrm{Log} \; \omega\) to analytically continue the function into the complex plane provided that \(-\pi < \mathrm{Arg} (\omega) \leq \pi \) with \(|\mathrm{Arg} (\omega)| \neq \pi/2 \).

There are two regularized limits to evaluate. Since the secant function is analytic at the origin, the first regularized limit reduces to the usual Cauchy limit. Expressing it in terms of Euler numbers \cite[4.19.5]{NIST}, we have
\begin{align}
\label{d-sec2}
    \lim^\times_{\nu \to 0} \omega^{-\nu} \frac{\mathrm{d}^j}{\mathrm{d}\nu^j} \sec \left( \frac{\pi \nu}{2} \right)
    &= \lim_{\nu \to 0} \omega^{-\nu} \frac{\mathrm{d}^j}{\mathrm{d}\nu^j} \sum_{k=0}^\infty \frac{(-1)^k E_{2k}}{(2k)!} \left( \frac{\pi}{2} \right)^{2k} \nu^{2k} \\
    &= \begin{cases}
        (-1)^j \left( \frac{\pi}{2} \right)^{2j} E_{2j}, & \text{if \(j\) is even}, \\
        0, & \text{if \(j\) is odd}.
    \end{cases} \nonumber
\end{align}
Substituting equation \eqref{d-sec2} into equation \eqref{d-sec} and simplifying the resulting terms yields equation \eqref{reglim-sec}.

For the second regularized limit, we use the formula \cite[p. 9, 1.1.5.10]{Brychkov}
\begin{align}
\label{d-csc1}
    \frac{\mathrm{d}^j}{\mathrm{d}\nu^j} \csc \left( \frac{\pi \nu}{2} \right)
    &= (-1)^{j+1} \frac{2 j!}{\pi \nu^{j+1}} + \frac{1}{2^{j+1} \pi} \left[ \psi^{(j)} \left( \frac{2+\nu}{4} \right) - (-1)^j \, \psi^{(j)} \left( \frac{2-\nu}{4} \right) \right] \\
    & \hspace{3cm} - \frac{1}{2^{2j+1} \pi} \left[ \psi^{(j)} \left( \frac{\nu}{4} \right) - (-1)^j \, \psi^{(j)} \left( \frac{\nu}{4} \right) \right], \nonumber
\end{align}
where \(\psi^{(n)}(z)\) denotes the polygamma function of order \(n\). Substituting equation \eqref{d-csc1} into equation \eqref{d-csc} decomposes the regularized limit into three terms. By linearity of the regularized limit, the first term evaluates, using equation \eqref{lemma2.1-001}, to
\begin{align}
    (-1)^{j+1} \frac{2 j!}{\pi} \lim^\times_{\nu \to 0} \frac{\omega^{-\nu}}{\nu^{j+1}} = \frac{2}{\pi} \frac{(\mathrm{Log} \, \omega)^{j+1}}{j+1}.
\end{align}

The second term involves the polygamma function, which is analytic in the right half-plane, thereby reducing the regularized limit to the standard Cauchy limit. Using the representation \cite[5.15.3]{NIST}
\begin{align}
    \psi^{(j)}\left( \frac{1}{2} \right) = (-1)^{j+1} j! \left( 2^{j+1} - 1 \right) \zeta(j+1),
\end{align}
where \(\zeta(z)\) denotes the Riemann zeta function, the regularized limit of the second term is proportional to
\begin{align}
    & \frac{1}{2^{j+1} \pi} \lim^\times_{\nu \to 0} \omega^{-\nu} \left[ \psi^{(j)}\left( \frac{2+\nu}{4} \right) - (-1)^j \psi^{(j)}\left( \frac{2-\nu}{4} \right) \right] \\
    & \hspace{2cm} = \frac{(-1)^j}{2^{2j+1} \pi} \left( 1 - (-1)^j \right) \left( 1 - 2^{j+1} \right) j! \, \zeta(j+1). \nonumber
\end{align}
Note that only odd values of \(j\) contribute to the summation.

The evaluation of the third term requires a polygamma function at \(\nu = 0\). A useful asymptotic expansion for \(z \to 0\) is given by \cite{mathematica001}
\begin{align}
\label{prth4.2-006}
    \psi^{(j)}(z) = \frac{(-1)^{j-1} j!}{z^{j+1}} + (-1)^{j-1} j! \, \zeta(j+1) \left( 1 + \mathcal{O}(z) \right), \quad j = 1, 2, 3, \ldots,
\end{align}
which, combined with equation \eqref{lemma2.1-001}, yields the regularized limit of the third term:
\begin{align}
    & -\frac{1}{2^{2j+1} \pi} \lim^\times_{\nu \to 0} \omega^{-\nu} \left[ \psi^{(j)}\left( \frac{\nu}{4} \right) - (-1)^j \psi^{(j)}\left( \frac{\nu}{4} \right) \right] \\
    & \hspace{2cm} = -\frac{4}{\pi} \frac{(\mathrm{Log} \, \omega)^{j+1}}{j+1} + \frac{(-1)^j}{2^{2j+1} \pi} \left( 1 - (-1)^j \right) j! \, \zeta(j+1). \nonumber
\end{align}

Substituting all evaluated regularized limits, simplifying the resulting terms, and applying the representation of the Riemann zeta function in terms of Bernoulli numbers \cite[25.6.2]{NIST},
\begin{align}
\label{th4.2-00}
    \zeta(2j) = \frac{(-1)^{j-1} 2^{2j-1} \pi^{2j}}{(2j)!} B_{2j},
\end{align}
leads to the final expression stated in equation \eqref{th4.2-001} of Theorem \ref{theorem4.2}.

We now proceed to verify that the infinite series in equation~\eqref{th4.2-001} converges absolutely. Let us consider the case where \(a < \infty\). There are two possible scenarios: either \(a < \rho_o\) or \(a > \rho_o\). We begin with the case \(\rho_o < a < \infty\). In this setting, the infinite series in equation~\eqref{th4.2-001} can be bounded as follows:
\begin{align}
    \label{prth4.2-001}
    \Biggl |\sum_{k=0}^\infty (-1)^k \omega^{2k}  \bbint{0}{a} \frac{k(t) \ln^n{t}}{t^{2k+2}} \, \mathrm{d}t \Biggr | 
    \leq \sum_{k=0}^\infty |\omega|^{2k}   \Biggl |\bbint{0}{a} \frac{k(t) \ln^n{t}}{t^{2k+2}} \, \mathrm{d}t \Biggr |.
\end{align}
According to the results in~\cite{galapon2023}, the finite-part integral in equation~\eqref{prth4.2-001} satisfies the following inequality:
\begin{align}
    \label{prth4.2-002}
    \Biggl | \bbint{0}{a} \frac{k(t) \ln^n{t}}{t^{2k+2}} \, \mathrm{d}t \Biggr | 
    \leq |c_{2k+1}| \frac{|\ln{\epsilon}|^{n+1}}{n+1} + \frac{M_o(a,\epsilon)}{\epsilon^{2k+1}},
\end{align}
where the function \(k(t)\) is expanded as \(k(t) = \sum_{l=0}^\infty c_l t^l\), and \(M_o(a, \epsilon)\) is given by
\begin{align}
    M_o(a,\epsilon) = \int_\epsilon^a \frac{|k(t)\ln^n{t}|}{t} \, \mathrm{d}t 
    + \sum_{j=0}^n \frac{|\ln{\epsilon}|^\epsilon}{j!} \sum_{\substack{l = 0 \\ l \neq 2k+1}}^\infty |c_l| \epsilon^l.
\end{align}
Substituting inequality~\eqref{prth4.2-002} into equation~\eqref{prth4.2-001} yields
\begin{align}
    \label{prth4.2-003}
    \Biggl |\sum_{k=0}^\infty (-1)^k \omega^{2k}  \bbint{0}{a} \frac{k(t) \ln^n{t}}{t^{2k+2}} \, \mathrm{d}t \Biggr | 
    \leq \frac{|\ln{\epsilon}|^{n+1}}{n+1} \sum_{k=0}^\infty |c_{2k+1} \omega^{2k}| 
    + \frac{M_o(a, \epsilon)}{\epsilon} \sum_{k=0}^\infty \left|\frac{\omega}{\epsilon}\right|^{2k}.
\end{align}
For the series in equation~\eqref{prth4.2-003} to converge absolutely, both terms must converge. The first term converges provided that \(|\omega| < \rho_o\), while the second term converges if \(|\omega| < \epsilon\). Therefore, the overall series converges absolutely if \(|\omega| < \rho_o\).

The remaining case \(a < \rho_o\) can be addressed using the same approach employed in the proof of Theorem~\ref{theorem4.1}. In this situation, it is important to note that the function \(k(t)\) is entire and therefore admits a power series representation. Consequently, it is valid to perform term-by-term integration followed by summation.

Now, consider the case when \(a = \infty\). As in the case \(\nu \neq 0\), the absolute convergence of the infinite series in equation \eqref{th4.2-001} must have a definite value, assuming that the generalized Stieltjes transform in equation \eqref{sec4.2-003} exists for \(\nu = 0\). This implies that Theorem~\ref{theorem4.2} remains valid as \(a \to \infty\).

Applying the same reasoning used in Theorem~\ref{theorem4.1} to prove the asymptotic relation in equation \eqref{th4.1-002-asym}, rewrite equation \eqref{th4.2-001} as
\begin{align}
\label{asymp-2}
    \int_0^a \frac{k(t) \ln^n t}{\omega^2 + t^2} \, \mathrm{d}t 
    = \biggl( \bbint{0}{a} \frac{k(t) \ln^n t}{t^{2}} \, \mathrm{d}t + O(\omega) \biggr) + \bigl( k(0) + O(i\omega) \bigr) \, \Delta_{n_1}(\nu, \omega),
\end{align}
as \(\omega \to 0\). From equation \eqref{reglim-sec}, we have \(\Delta_{n_1}(\nu, \omega) = O\bigl(\omega^{-1} \mathrm{Log}^{n} \omega \bigr)\) as \(\omega \to 0\). Therefore, the second term dominates the first and third terms in equation \eqref{asymp-2}, which establishes equation \eqref{th4.1-002-asym}. The proof presented under this theorem follows the same structure detailed in \cite{galapon2023}.
\end{proof}



\end{document}